\documentclass{amsart}
\usepackage{graphicx,verbatim, amsmath, amssymb, amsthm, amsfonts, epsfig, ifthen,mathtools,mathabx,enumerate}%,epstopdf}	
\newtheorem{proposition}{Proposition}[section]
\newtheorem{theorem}[proposition]{Theorem}

\newtheorem{lemma}[proposition]{Lemma}
\newtheorem{prop}[proposition]{Proposition}
\newtheorem{cor}[proposition]{Corollary}

\theoremstyle{definition}

\theoremstyle{remark}
\newtheorem{remark}[proposition]{Remark}

\numberwithin{equation}{section}
\usepackage[usenames]{color}

\usepackage[usenames]{color}

\newcounter{margincounter}
\newcommand{\newword}[1]{\textbf{\emph{#1}}}

\newcommand{\integers}{\mathbb Z}
\newcommand{\rationals}{\mathbb Q}

\newcommand{\reals}{\mathbb R}

\newcommand{\cl}{\operatorname{cl}}
\newcommand{\cw}{\operatorname{cw}}
\newcommand{\ccw}{\operatorname{ccw}}
\newcommand{\sgn}{\operatorname{sgn}}
\newcommand{\vsgn}{\mathbf{sgn}}

\newcommand{\set}[1]{{\left\lbrace #1 \right\rbrace}}

\newcommand{\A}{{\mathcal A}}

\newcommand{\F}{{\mathcal F}}

\renewcommand{\th}{^\mathrm{th}}
\newcommand{\nd}{^\mathrm{nd}}

\newcommand{\1}{{\hat{1}}}

\newcommand{\g}{\mathbf{g}}
\renewcommand{\d}{\mathbf{d}}
\renewcommand{\b}{\mathbf{b}}
\renewcommand{\k}{\mathbf{k}}
\renewcommand{\a}{\mathbf{a}}
\newcommand{\e}{\mathbf{e}}
\newcommand{\x}{\mathbf{x}}

\renewcommand{\t}{\mathbf{t}}
\renewcommand{\v}{\mathbf{v}}
\renewcommand{\u}{\mathbf{u}}
\newcommand{\s}{\mathbf{s}}
\newcommand{\w}{\mathbf{w}}
\newcommand{\tB}{\tilde{B}}

\newcommand{\M}{\mathcal{M}}

\renewcommand{\S}{\mathbf{S}}
\renewcommand{\M}{\mathbf{M}}
\renewcommand{\r}{\mathtt{r}}
\renewcommand{\t}{\mathtt{t}}
\renewcommand{\c}{\mathtt{c}}

\title{Universal geometric coefficients for the once-punctured torus}
\author{Nathan Reading}
\thanks{This material is based upon work partially supported by the Simons Foundation under Grant Number 209288 and by the National Science Foundation under Grant Number DMS-1101568.}
\subjclass[2010]{13F60, 57Q15}

\begin{document}

\begin{abstract}
We construct universal geometric coefficients, over $\integers$, $\rationals$, and $\reals$, for cluster algebras arising from the once-punctured torus.
We verify that the once-punctured torus has a property called the Null Tangle Property.
The universal geometric coefficients over $\integers$ and $\rationals$ are then given by the shear coordinates of certain ``allowable'' curves in the torus.
The universal geometric coefficients over $\reals$ are given by the shear coordinates of allowable curves together with the normalized shear coordinates of certain other curves each of which is dense in the torus.
We also construct the mutation fan for the once-punctured torus and recover a result of N\'{a}jera on $\g$-vectors.
\end{abstract}
\maketitle

\setcounter{tocdepth}{1}
\tableofcontents

\section{Introduction}\label{intro}
Given an exchange matrix $B$, a universal geometric cluster algebra for $B$ is a cluster algebra that is universal, in the sense of coordinate specialization, among cluster algebras of geometric type (broadly defined in the sense of \cite{universal}) associated to $B$.
A universal geometric cluster algebra is specified by an extended exchange matrix (again broadly defined) whose coefficient rows are called the universal geometric coefficients for $B$.
The polyhedral geometry and ``mutation-linear algebra'' of these universal geometric coefficients is worked out in \cite{universal}.
In the case where $B$ arises from a marked surface, universal geometric coefficients and the closely related mutation fans can be approached through a variant of the laminations that appear in \cite{cats2}.
In \cite{unisurface}, this approach yields a construction of the mutation fan for all marked surfaces except the once-punctured surfaces without boundary, and a construction of universal geometric coefficients for a smaller family of surfaces.

In this paper, we carry out the construction of universal geometric coefficients and the mutation fan for the simplest case not handled in \cite{unisurface}:  the once-punctured torus.
The essential step in the constructions of this paper is to establish the Null Tangle Property for the once-punctured torus.
The Null Tangle Property implies that the universal geometric coefficients (over $\integers$ or $\rationals$) are given by the shear coordinates of certain curves.
Universal geometric coefficients over $\rationals$ were already known to exist, by a non-constructive proof \cite[Corollary~4.7]{universal}, but even the existence of universal geometric coefficients over $\integers$ was not previously known for the once-punctured torus.
Our proof of the Null Tangle Property uses tools that are not surprising in the context of the once-punctured torus.
Similar ideas have appeared, for example, in \cite{BonahonZhu,FrohGel,Minsky,Najera,Series}, and our treatment draws on ideas from these earlier works.
We realize shear coordinates of curves in the once-punctured torus in terms of Farey triples.

We take the results for the once-punctured torus farther than the results for other surfaces in \cite{unisurface} by also computing universal geometric coefficients over~$\reals$.
The universal geometric coefficients over $\reals$ are given by the (normalized) shear coordinates of a larger collection of curves, some of which are dense in the torus.

The following theorem gives the universal geometric coefficients explicitly.
\begin{theorem}\label{markov main}
Let $B=\begin{bsmallmatrix*}[r]0&2&-2\\-2&0&2\\2&-2&0\\\end{bsmallmatrix*}$.
A universal extended exchange matrix for $B$ over $\integers$ or $\rationals$ has the following coefficient rows:
\begin{enumerate}
\item \label{markov gs}
The three cyclic permutations of $[1-b,\,a+1,\,b-a-1]$, for all (possibly infinite) positive slopes with standard form $b/a$.
\item \label{markov opp gs}
The three cyclic permutations of $[-1-b,\,a-1,\,b-a+1]$, for all finite nonnegative slopes with standard form $b/a$.
\item \label{markov equatorial}
All integer vectors $[x,\,y,\,z]$ with $x+y+z=0$ such that $x$, $y$, and $z$ have no common factors.
\end{enumerate}
A universal extended exchange matrix for $B$ over $\reals$ has the coefficient rows described in \eqref{markov gs} and \eqref{markov opp gs} above, as well as 
\begin{enumerate}
\item[($3'$)] \label{extra equatorial}
Exactly one nonzero vector in $\rho$ for each ray $\rho$ contained in the plane given by $x+y+z=0$.
\end{enumerate}
\end{theorem}

The mutation fan for $B=\begin{bsmallmatrix*}[r]0&2&-2\\-2&0&2\\2&-2&0\\\end{bsmallmatrix*}$ is described in Theorem~\ref{torus FB} and illustrated in Figure~\ref{torus FB fig}.
\begin{figure}[ht]
\scalebox{0.999}{\includegraphics{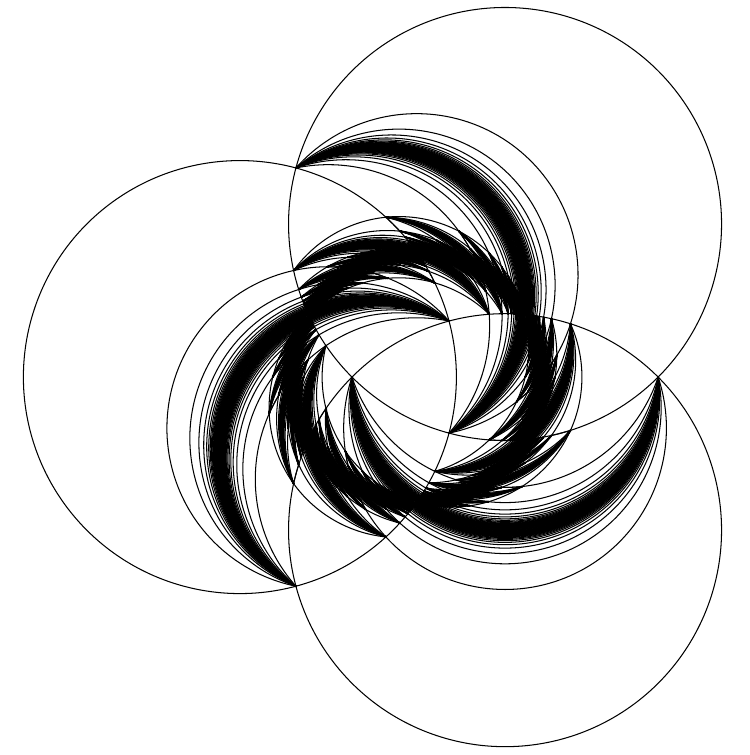}}
\begin{picture}(0,0)(0,-181)
\put(26.5,-9){\small$\e_2$}
\put(26.5,30){\small$\e_3$}
\put(-7,10.5){\small$\e_1$}
\end{picture}
\caption{The mutation fan for the once-punctured torus}
\label{torus FB fig}
\end{figure}
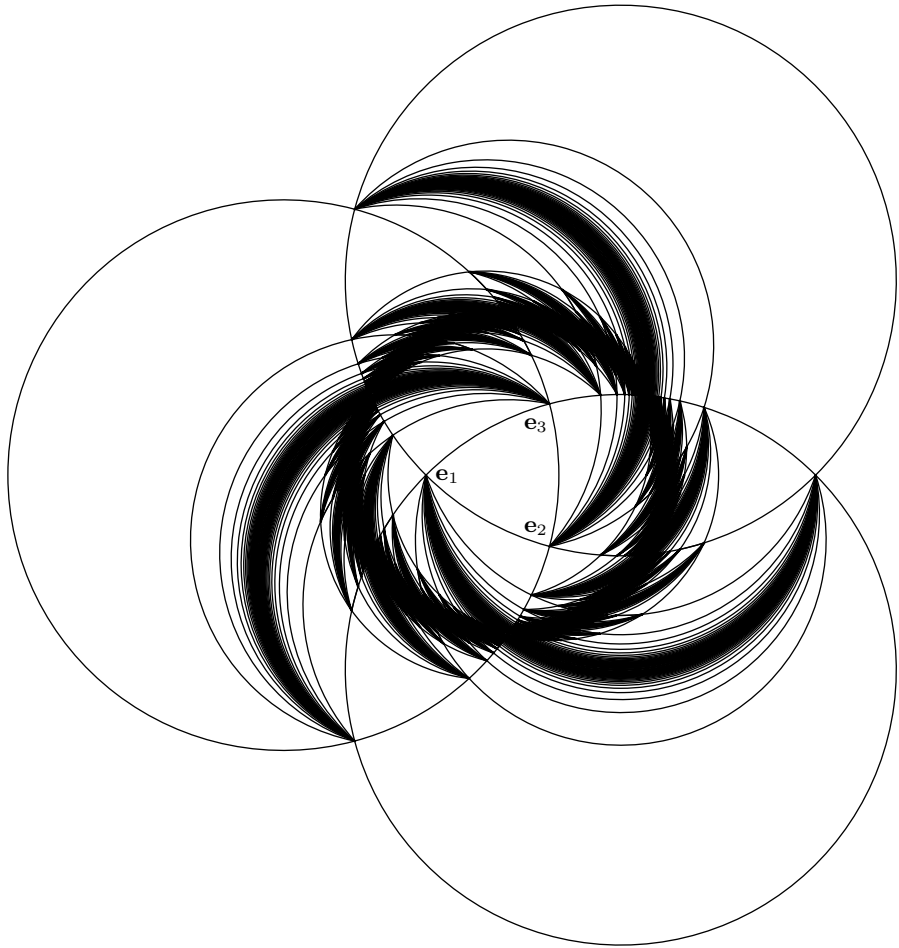
The figure is interpreted as follows:
Each cone of the mutation fan intersects the unit sphere about the origin in a point, circular arc, or spherical triangle.
The figure pictures the images of these points, arcs, and triangles under a stereographic projection of the unit sphere to the plane.
The projections of the standard basis vectors $\e_1$, $\e_2$, and $\e_3$ are indicated.

As a corollary (Corollary~\ref{markov g}) to Theorem~\ref{torus FB}, we recover a result of N\'{a}jera \cite{Najera} describing the $\g$-vectors of cluster variables for~$B^T$.
N\'{a}jera's result was the starting point for our investigation of the mutation fan for the once-punctured torus.
In two short sections at the end of the paper, we briefly discuss denominator vectors, and then briefly discuss how the results of this paper have been extended to the four-punctured sphere \cite{unisphere} and might extend to other tubular cluster algebras in the sense of \cite{tubular1,tubular2}.

\section{Universal geometric coefficients}
In this section, we briefly review background on universal geometric cluster algebras.
Rather than give complete details, we move quickly to reframe the problem of constructing universal geometric coefficients as the problem of finding a basis, and then quote results that describe how to find a basis in the case of surfaces.
We also introduce the closely related notion of the mutation fan.
Details can be found in \cite{universal}.
Additional background on geometric cluster algebras is found for example in \cite{ca4}, but we work with the broader definition originated in~\cite{universal}.

We choose an \newword{underlying ring} $R$ to be $\integers$, $\rationals$, or $\reals$, fix an integer $n>0$, and let $I$ be an indexing set of arbitrary cardinality.
An  \newword{extended exchange matrix} $\tB$ is a collection of vectors in $R^n$ called the \newword{rows} of $\tB$, indexed by the disjoint union $[n]\cup I$.
We think of $\tB$ as an ``$(n+|I|)\times n$ matrix,'' although since $I$ may be infinite or even uncountable, it may not be a matrix in the usual sense.
However, the rows of $\tB$ indexed by $[n]$ constitute an $n\times n$ matrix denoted by $B$.
This \newword{exchange matrix} $B$ is in general required to have integer entries and to be skew-symmetrizable, but the matrix under consideration in this paper has a stronger property than skew-symmetrizability:  it is skew-symmetric in the usual sense.
The rows of $\tB$ indexed by $I$ are called \newword{coefficient rows}.

An extended exchange matrix defines a \newword{cluster algebra of geometric type}, denoted by $\A_R(\tB)$.
The details of the definition are in \cite{universal}; we briefly sketch it here.
One takes the coefficient semifield to be a tropical semifield over tropical variables indexed by $I$, with exponents on the tropical variables taking values in $R$, rather than only in $\integers$.
One then uses $\tB$ to define initial coefficients just as in the \emph{geometric type} construction from \cite{ca4} and takes $B$ and these coefficients as the input to the general construction of cluster algebras.
The cluster algebra $\A_R(\tB)$ may not be of geometric type in the sense of \cite{ca4}, for two reasons:  
First, there may be infinitely many coefficient rows of $\tB$ (or equivalently infinitely many tropical variables), and second, the coefficient rows of $\tB$ may have non-integer entries.
In \cite[Definition~2.8]{universal}, the definition of cluster algebras of geometric type is broadened to include exactly the cluster algebras of the form $\A_R(\tB)$ as described here.

For those less familiar with cluster algebras:
A \newword{seed} is a pair $(\tB,\x)$ consisting of an extended exchange matrix $\tB$ and an $n$-tuple $\x=(x_1,\ldots,x_n)$, called a \newword{cluster}.
The $x_i$ are rational functions called \newword{cluster variables}.
Given a seed and an index $k\in[n]$, there is a notion of \newword{seed mutation} that exchanges $(\tB,\x)$ for a new seed $\mu_k(\tB,\x)=(\tB',\x')$.
The matrix $\tB'$ is obtained from $\tB$ by matrix mutation, which we define below.
Writing $x'_1,\ldots,x'_n$ for the cluster variables in the new cluster, we have $x'_j=x_j$ for each $j\neq k$, while $x'_k$ and $x_k$ are related by an \newword{exchange relation} which writes $x'_k$ as a rational function of the cluster variables in $\x$, with coefficients given by the coefficient rows of $\tB$.
We perform all possible sequences of mutations to obtain a (typically infinite) collection of seeds.
Each cluster variable in each seed arising in this way is a rational function in the original cluster variables $x_1,\ldots,x_n$.
The cluster algebra $\A_R(\tB)$ is the algebra (i.e.\ the subalgebra of the field of rational functions) generated by all of the cluster variables in all of the seeds.

Matrix mutation is defined as follows.
Given an extended exchange matrix $\tB$, we write $b_{ij}$ for the $j\th$ entry in the row of $\tB$ indexed by $i$.
Thus $i$ is in the disjoint union $[n]\cup I$ and $j$ is in $[n]$.
For $k\in[n]$, the mutation of $\tB$ at index $k$ is the matrix $\mu_k(\tB)=\tB'$ with entries $b'_{ij}$ given by
\begin{equation}\label{b mut}
b_{ij}'=\left\lbrace\!\!\begin{array}{ll}
-b_{ij}&\mbox{if }i=k\mbox{ or }j=k;\\
b_{ij}+\sgn(b_{kj})\,[b_{ik}b_{kj}]_+&\mbox{otherwise,}
\end{array}\right.
\end{equation}
where $\sgn(0)=0$ and $\sgn(a)=a/|a|$ for $a\neq 0$.
We are interested in sequences of mutations, and we establish notation for such sequences.
Given a sequence $\k=k_q,\ldots,k_1$ of integers in~$[n]$, the notation $\mu_\k$ stands for $\mu_{k_q}\circ\mu_{k_{q-1}}\circ\cdots\circ\mu_{k_1}$.

Suppose $\tB$ and $\tB'$ are extended exchange matrices having the same exchange matrix $B$.
A \newword{coefficient specialization} from $\A_R(\tB)$ to $\A_R(\tB')$ is a ring homomorphism taking each cluster variable in $\A_R(\tB)$ to the corresponding cluster variable in $\A_R(\tB')$, taking the coefficients in each exchange relation to the corresponding coefficients, and satisfying some additional technical requirements.
A cluster algebra $\A_R(\tB)$ is a \newword{universal geometric cluster algebra} over $R$ for the exchange matrix $B$ if, for any other geometric cluster algebra $\A_R(\tB')$ with exchange matrix $B$, there is a unique coefficient specialization from $\A_R(\tB)$ to $\A_R(\tB')$.
In this case, $\tB$ is called a \newword{universal extended exchange matrix} over $R$ and the coefficient rows of $\tB$ are called \newword{universal geometric coefficients} for $B$ over $R$.
Rather than give the details of the definition of coefficient specialization, we now explain how to understand universal geometric coefficients for $B$ directly.

An exchange matrix $B$ defines, by matrix mutation, a collection of maps $\eta_\k^B:\reals^n\to\reals^n$ that we call \newword{mutation maps}.
Given $\a=(a_1,\ldots,a_n)\in\reals^n$, create an extended exchange matrix $\begin{bsmallmatrix}B\\\a\end{bsmallmatrix}$ with exchange matrix $B$ and a single coefficient row~$\a$.
Given a sequence $\k$ of integers in~$[n]$, the vector $\eta_\k^B(\a)$ is defined to be the coefficient row of the matrix $\mu_\k(\begin{bsmallmatrix}B\\\a\end{bsmallmatrix})$,
The map $\eta_\k^B$ is a piecewise-linear homeomorphism from $\reals^n$ to $\reals^n$, with inverse $\eta^{\mu_\k(B)}_{k_1,\ldots,k_q}$.
When $\k$ is the singleton sequence $k$, the vector $(a_1',\ldots,a_n')=\eta_k^B(a_1,\ldots,a_n)$ is given by 
\begin{equation}\label{mutation map def}
a'_j=\left\lbrace\begin{array}{ll}
-a_k&\mbox{if }j=k;\\
a_j+a_kb_{kj}&\mbox{if $j\neq k$, $a_k\ge 0$ and $b_{kj}\ge 0$};\\
a_j-a_kb_{kj}&\mbox{if $j\neq k$, $a_k\le 0$ and $b_{kj}\le 0$};\\
a_j&\mbox{otherwise.}
\end{array}\right.
\end{equation}

Given a collection of vectors $(\v_i:i\in S)$ in $\reals^n$ indexed by a finite set $S$, the formal expression $\sum_{i\in S}c_i\v_i$ is a \newword{$B$-coherent linear relation with coefficients in $R$} if each $c_i$ is in $R$ and if the equalities
\begin{eqnarray}
\label{linear eta}
&&\sum_{i\in S}c_i\eta^B_\k(\v_i)=\mathbf{0},\mbox{ and}\\
\label{piecewise eta}
&&\sum_{i\in S}c_i\mathbf{min}(\eta^B_\k(\v_i),\mathbf{0})=\mathbf{0}
\end{eqnarray}
hold for every finite sequence $\k=k_q,\ldots,k_1$ of integers in~$[n]$.
The symbol $\mathbf{min}$ stands for componentwise minimum.
Since \eqref{linear eta} is required also for the empty sequence $\k$, a $B$-coherent linear relation is also a linear relation in the usual sense. 

A collection $(\b_i:i\in I)$ of vectors in $R^n$, indexed by an arbitrary set $I$ is an \newword{$R$-spanning set for $B$} if for all $\a\in R^n$, there exists a finite subset $S\subseteq I$ and elements $(c_i:i\in S)$ of $R$ such that $\a-\sum_{i\in S}c_i\b_i$ is a $B$-coherent linear relation.
The collection $(\b_i:i\in I)$ is an \newword{$R$-independent set for $B$} if for every $B$-coherent linear relation $\sum_{i\in S}c_i\b_i$ with $S\subseteq I$, each $c_i$ is zero.
An \newword{$R$-basis} for $B$ is a collection $(\b_i:i\in I)$ that is both an $R$-independent set for $B$ and an $R$-spanning set for $B$.
The following theorem is \cite[Theorem~4.4]{universal}.

\begin{theorem}\label{basis univ}
Let $\tB$ be an extended exchange matrix with entries in $R$.
Then $\tB$ is universal over $R$ if and only if the coefficient rows of $\tB$ are an $R$-basis for $B$.
\end{theorem}

For a vector $\a=(a_1,\ldots,a_n)$ in $\reals^n$ we define $\vsgn(\a)$ to be $(\sgn(a_1),\ldots,\sgn(a_n))$.
Using mutation maps and the function $\vsgn$, we define an equivalence relation on $\reals^n$.
Set $\a_1\equiv^B\a_2$ if and only if $\mathbf{sgn}(\eta^B_\k(\a_1))=\mathbf{sgn}(\eta^B_\k(\a_2))$ for every sequence $\k$ of integers in~$[n]$.
The $\equiv^B$-equivalence classes are called \newword{$B$-classes} and their closures are called $B$-cones.
Each $B$-cone is a closed convex cone.
The following proposition is \cite[Proposition~5.3]{universal}.

\begin{prop}\label{linear}
Every mutation map $\eta_\k^B$ is linear on every $B$-cone.
\end{prop}

The collection $\F_B$, consisting of all $B$-cones together with all their faces is called the \newword{mutation fan} for $B$, in light of the following theorem  \cite[Theorem~5.13]{universal}.  
\begin{theorem}\label{fan}
The collection $\F_B$ is a complete fan.
\end{theorem}

A \newword{positive $R$-basis} for $B$ is an $R$-independent set for $B$ such that for any $\a\in R^n$, there is a $B$-coherent linear relation $\a-\sum_{i\in S}c_i\b_i$ with each $c_i$ nonnegative.
A positive $R$-basis for $B$ is in particular an $R$-basis for $B$.
For each $B$ and $R$, there is at most one positive $R$-basis for $B$, up to scaling each basis element by a positive unit in $R$.
(See  \cite[Proposition~6.2]{universal}.)
In the case $R=\reals$, there is a direct connection between the mutation fan and the problem of constructing a positive $R$-basis for $B$, described in the following propositions, which are \cite[Corollary~6.13]{universal} and \cite[Proposition~6.14]{universal}.

\begin{prop}\label{pos reals FB}
If a positive $\reals$-basis for $B$ exists, then $\F_B$ is simplicial.  
The basis consists of exactly one vector in each ray of $\F_B$.
\end{prop}

\begin{prop}\label{pos reals FB converse}
A collection consisting of exactly one nonzero vector in each ray of $\F_B$ is a positive $\reals$-basis for $B$ if and only if it is an $\reals$-independent set for~$B$.
\end{prop}

For $R\neq\reals$, a similar statement can be made in terms of the $R$-part (e.g.\ the rational part) of $\F_B$. 
See \cite[Definition~6.9]{universal} and \cite[Corollary~6.12]{universal}.

\section{The once-punctured torus}
The purpose of this section is to review background material on (universal) cluster algebras from surfaces.
However, since this paper is focused on a particular surface, the once-punctured torus, we are able to simplify some definitions and results.
For the general definitions and results, see  \cite{cats1,cats2,unisurface}.

In general, a marked surface is $(\S,\M)$, where $\S$ is a surface and $\M$ is some set of points on $\S$, satisfying certain rules.
Here we take $\S$ to be a torus and $\M=\set{p}$ where $p$ is some point in $\S$ and write $(\S,p)$ rather than $(\S,\set{p})$.
The point $p$ is referred to as a puncture, so that $(\S,p)$ is the once-punctured torus.
An \newword{arc} in $(\S,p)$ is a curve in $\S$, up to isotopy relative to $\set{p}$, with both endpoints at $p$, but otherwise not containing $p$ and having no self-intersections.
We exclude the arc that bounds an unpunctured monogon.
Two arcs $\alpha$ and $\gamma$ are \newword{compatible} if (there is some isotopy representative of each such that) the two do not intersect except at $p$.
The \newword{arc complex} is the abstract simplicial complex whose vertices are the arcs and whose faces are the sets of pairwise compatible arcs.
A \newword{triangulation} of $(\S,p)$ is a maximal face of the arc complex.
A triangulation contains three arcs and cuts $\S$ into two triangles.
Each arc in a triangulation is an edge of exactly two triangles in the triangulation.
A \newword{flip} is the operation of removing an arc from a triangulation and forming a new triangulation by inserting an arc that forms the other diagonal of the resulting quadrilateral.
A \newword{tagged arc} in $(\S,p)$ is a arc in $\S$, either with both ends of the arc designated (or ``tagged'') \newword{plain} or with both tagged \newword{notched}.
%(This is a special case of a definition that makes more sense in general.)
In general, one makes a suitable definition of compatibility of tagged arcs and works with the ``tagged arc complex.''
For $(\S,p)$, however, we can (and indeed are forced to) work with the ordinary arc complex.
(See \cite[Proposition~7.10]{cats1}.)

Each triangulation $T$ of $(\S,p)$ has a \newword{signed adjacency matrix} $B(T)$, defined in \cite[Definition~5.15]{cats2}.
This is a matrix indexed by the arcs in $T$.
The entry $b_{\alpha\gamma}$ is the sum, over the two triangles of $T$, of a $1$ if $\gamma$ immediately follows $\alpha$ in a clockwise path around the triangle or $-1$ if $\gamma$ immediately follows $\alpha$ in a counterclockwise path around the triangle.
For the rest of the paper, the symbol $B$ will stand for $\begin{bsmallmatrix*}[r]0&2&-2\\-2&0&2\\2&-2&0\\\end{bsmallmatrix*}$.
For any triangulation $T$ of $(\S,p)$, the matrix $B(T)$ is $\pm B$.

An \newword{allowable curve} in $(\S,p)$ is a non-self-intersecting curve in $\S\setminus\set{p}$ that either 
\begin{itemize}
\item is closed and not contractible in $\S$, or
\item spirals into p at both ends.
\end{itemize}
In the once-punctured torus, a non-closed allowable curve spirals into the same puncture at both ends, and therefore spirals in the same direction at both ends.

Two allowable curves are \newword{compatible} if they do not intersect.
In \cite{unisurface}, the notions of allowable curves and compatibility refer to \newword{quasi-laminations}, a modification of the \newword{rational unbounded measured laminations} that figure in~\cite{cats2}.
For the once-punctured torus, the two notions of laminations coincide.

Given an allowable curve $\lambda$ and a triangulation $T$, the \newword{shear coordinate vector} $\b(T,\lambda)$ of $\lambda$ with respect to $T$ is a vector with entries $b_\gamma(T,\lambda)$ indexed by arcs $\gamma$ in~$T$.
The entry $b_\gamma(T,\lambda)$ is defined as follows.
Up to isotopy, we can assume that $\lambda$ does not cross any arc twice consecutively in opposite directions.
The coordinate $b_\gamma(T,\lambda)$ is the sum, over the intersections of $\lambda$ with $\gamma$, of a number in $\set{-1,0,1}$.
This number is determined by how $\lambda$ intersects the two triangles having $\gamma$ as an edge.
It is $1$ or $-1$ if the intersection is as shown in Figure~\ref{shear fig}, and otherwise it is~$0$.
\begin{figure}[ht]
\raisebox{42 pt}{$+1$}\,\,\includegraphics{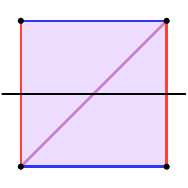}\qquad\qquad\includegraphics{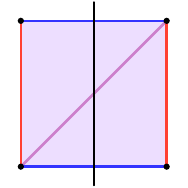}\,\,\raisebox{42 pt}{$-1$}
\caption{Computing shear coordinates}
\label{shear fig}
\end{figure}
The arc $\gamma$ is the diagonal of the quadrilateral shown, and opposite sides of the quadrilateral coincide.

A \newword{(weighted) tangle (of curves)} in $(\S,p)$ is a finite collection of allowable curves $\lambda$, each having an integer weight $w_\lambda$, with no requirement of compatibility between curves.
The shear coordinates $\b(T,\Xi)$ of a tangle $\Xi$ are given by the sum $\sum_{\lambda\in\Xi}w_\lambda \b(T,\lambda)$.
A \newword{null tangle} is a tangle with $b(T,\Xi)=\mathbf{0}$ for all triangulations $T$.
Proposition~\ref{one positive or one negative}, below, is \cite[Proposition~7.11]{unisurface}, specialized to the once-punctured torus.
We point out, however, an unfortunate typo in \cite[Proposition~7.11]{unisurface}:  It should begin ``Let $\Xi$ be a \emph{null} tangle\ldots,'' but the word ``null'' is omitted in~\cite{unisurface}.
The proposition itself is the specialization to surfaces of \cite[Proposition~4.12]{universal}.
See also \cite[Proposition~7.9]{unisurface}, quoted in this paper as Proposition~\ref{null tangle is B coher rel}.

\begin{prop}\label{one positive or one negative}
Let $\Xi$ be a null tangle in $(\S,p)$.
Suppose for some triangulation $T$, for some arc $\gamma$ in $T$, and for some curve $\lambda$ in $\Xi$ that $b_\gamma(T,\lambda)$ is strictly positive (resp. strictly negative) and that $b_\gamma(T,\nu)$ is nonpositive (resp. nonnegative) for every other $\nu\in\Xi$.
Then $w_\lambda=0$.
\end{prop}

The \newword{support} of a tangle is the set of curves having nonzero weight.
A tangle is \newword{trivial} if all weights are zero.
A marked surface has the \newword{Null Tangle Property} if every null tangle is trivial.
In Section~\ref{null tangle sec}, we prove the following theorem.
\begin{theorem}\label{torus null tangle}
The once-punctured torus has the Null Tangle Property.
\end{theorem} 
Both results quoted below in this section are obtained by combining Theorem~\ref{torus null tangle} with results of \cite{unisurface} that are conditional on the Null Tangle Property.
Although we do not prove Theorem~\ref{torus null tangle} until Section~\ref{null tangle sec}, we do not use it or either of the results below until Section~\ref{fan sec} and later.

%The Null Tangle Property is the crucial property of a marked surface with regard to universal geometric coefficients.
The specialization of \cite[Theorem~7.3]{unisurface} to the once-punctured torus says that the Null Tangle Property is equivalent to the following statement:
The shear coordinates of allowable curves form a positive $R$-basis for $B(T)$, where $T$ is any triangulation and $R$ is $\integers$ or $\rationals$.
Thus Theorem~\ref{torus null tangle} has the following corollary.
\begin{cor}\label{torus pos basis}
If $T$ is a triangulation of $(\S,p)$ and $R$ is $\integers$ or $\rationals$, then the shear coordinates $\b(T,\lambda)$ of allowable curves $\lambda$ constitute a positive $R$-basis for $B(T)$.
\end{cor}

Fix a triangulation $T$ of $(\S,p)$.
For each set $\Lambda$ of pairwise compatible allowable curves in $(\S,p)$, let $C_\Lambda$ be the nonnegative $\reals$-linear span of the vectors $\set{\b(T,\lambda):\lambda\in\Lambda}$.
Let $\F_\rationals(T)$ be the collection of all such cones $C_\Lambda$.
In \cite[Theorem~4.10]{unisurface}, this collection is shown to be a fan and, more specifically, to be the rational part of $\F_B$ in the sense of \cite[Definition~4.9]{unisurface}.
Rather than stating this definition, we summarize parts of \cite[Theorem~4.10]{unisurface} and its corollary \cite[Corollary~4.11]{unisurface} in the following proposition, which is stated in the special case of the once-punctured torus.
The hypotheses of \cite[Theorem~4.10]{unisurface} and \cite[Corollary~4.11]{unisurface} include the assumption that the marked surface has a property called the Curve Separation Property.
However, \cite[Corollary~7.12]{unisurface} states that the Null Tangle Property implies the Curve Separation Property, so by Theorem~\ref{torus null tangle}, these results apply to the once-punctured torus.

\begin{prop}\label{rat FB summary}
Let $T$ be a triangulation of $(\S,p)$.
The collection $\F_\rationals(T)$ is a rational, simplicial fan.
Each cone in $\F_\rationals(T)$ is contained in a cone of $\F_{B(T)}$.
Each rational $B(T)$-cone is a cone in $\F_\rationals(T)$.
The full-dimensional $B(T)$-cones are exactly the full-dimensional cones in~$\F_\rationals(T)$.
\end{prop}

\section{Arcs and allowable curves}\label{arc curve sec}
In this section, we prepare for the proof of Theorem~\ref{torus null tangle} by classifying arcs and allowable curves in $(\S,p)$.
Recall the usual construction of the torus as a quotient of its universal cover $\reals^2$:
We identify points $[x_1,\,y_1]$ and $[x_2,\,y_2]$ in $\reals^2$ if the fractional parts of $x_1$ and $x_2$ coincide and the fractional parts of $y_1$ and $y_2$ coincide.
We assume that the marked point $p$ is the image of $[0,\,0]$ under the covering map

\begin{prop}\label{torus arcs}
The arcs in $(\S,p)$ are the images, under the covering map, of straight line segments connecting the origin to nonzero integer points and containing no integer points in their interior.
%SinceSLC:  Fixed the next sentence, which used to read "Each image is a distinct arc up to isotopy." when in fact two segments map to each arc.
The map from segments to arcs is two-to-one, with antipodal pairs of segments mapping to the same arc.
%The segment to $[a,b]$ and the segment to $[c,d]$ map to the same arc, up to isotopy, if and only if $[c,d]=-[a,b]$.
%Each image is a distinct arc up to isotopy.  
\end{prop}
\begin{proof}
Arcs in $(\S,p)$ are in particular elements of the fundamental group of $\S$ based at the point $p$.
A non-trivial element of the fundamental group represents an arc in $\S$ if and only if (up to homotopy relative to $p$) it is non-self-intersecting except that the endpoints coincide.
The fundamental group of the torus consists of the projections of curves connecting the origin to points in $\integers^2$, and two fundamental group elements coincide if and only if they connect the origin to the same point.
The images of straight line segments described in the proposition are thus all distinct up to homotopy and it is immediate that each image is non-self-intersecting, except that the endpoints coincide.

It remains to show that every arc is represented by the projection of a line segment.
Let $\gamma$ be an arc and consider a lift $\bar\gamma$ of $\gamma$ having an endpoint at $[0,\,0]$.
Then $\bar\gamma$ is disjoint from all other lifts of $\gamma$ (i.e.\ from all integer translates of $\bar\gamma$), except for coinciding endpoints.
We break $\reals^2$ into unit squares with integer vertices and describe $\bar\gamma$, up to homotopy in $\reals^2\setminus\integers^2$, by the sequence of squares it visits.
We can assume that $\bar\gamma$ never enters a square and leaves through the same side.
(Otherwise, we continuously deform $\bar\gamma$ so that it doesn't enter the square, and we can do this without losing the property that $\bar\gamma$ is disjoint from all of its integer translates.)
Similarly, we can assume that after $\bar\gamma$ leaves $[0,\,0]$ and enters the interior of some square, it leaves that square through a side not containing $[0,\,0]$.
We can also assume the analogous condition at the other endpoint of $\bar\gamma$.
By symmetry, we can assume that $\bar\gamma$ leaves $[0,\,0]$, enters the square shared by $[0,\,0]$ and $[1,\,1]$ and leaves this square through the top side or right side or terminates at one of the vertices of the square.
In the latter case, we are done.
If $\bar\gamma$ leaves the square through the top or right side, then it enters some other square.
Because $\bar\gamma$ is disjoint from all its translates, we see that $\bar\gamma$ must either terminate at the top-right vertex of the new square or exit the new square through the top or right edge.
The same is true each time $\bar\gamma$ enters a new square:  it terminates in the top-right vertex or leaves through the top or right edge.
Thus $\bar\gamma$ determines a word in the letters $\r$ and $\t$, with $\r$ standing for exiting a square through the right and $\t$ standing for exiting through the right.

Let $\v$ be the endpoint of $\bar\gamma$ other than $[0,\,0]$, and consider the straight line segment $\lambda$ connecting $[0,\,0]$ to $\v$.
This line segment determines a word in the same way, except that in addition to $\r$ and $\t$ we allow a third letter $\c$ in the word for the case when $\lambda$ leaves a square through its top-right corner without terminating there.
%SinceSLC:  Replaced the following two sentences with something clearer.   :CLSecnis%
%If the words for $\bar\gamma$ and $\lambda$ coincide, then there is a homeomorphism of the square, homotopic to the identity, fixing corners, and agreeing on opposite sides, that takes the projection of $\bar\gamma$ to the projection of~$\lambda$.  
%Identifying opposite sides, we obtain an isotopy from $\gamma$ to the projection of $\lambda$.
If the words for $\bar\gamma$ and $\lambda$ coincide, then there is a homeomorphism of the plane, homotopic to the identity, fixing the integer lattice, and commuting with integer translations, that takes $\bar\gamma$ to~$\lambda$.  
Descending to the torus, we obtain an isotopy from $\gamma$ to the projection of $\lambda$.

If the two words are different, then we can still continuously deform $\bar\gamma$ to a non-self-intersecting polygonal path composed of straight segments with finite positive slopes connecting different sides of squares.
Concatenating $\bar\gamma$ with its translates by integer multiples of $\v$, we obtain the graph of a strictly increasing periodic piecewise-linear function.
Doing the same with $\lambda$, we obtain a strictly increasing linear function.
We reuse the symbols $\bar\gamma$ and $\lambda$ for these functions.
Since the two words disagree, either there is an integer point in the interior of $\lambda$ or there is an integer point between the two graphs.

%SinceSLC:  Clarified the following sentence.  We need to choose $\w$ to be the "first" (closest to $[0,0]$) integer 
% point in the interior of $\lambda$ or else we can say there is some integer $k>1$ such that $k\w=\v$.  :CLSecnis%
%First suppose there is an integer point $\w$ in the interior of $\lambda$.
First, suppose there is an integer point in the interior of $\lambda$, and choose $\w$ to be the integer point in the interior closest to $[0,0]$.
There is some integer $k>1$ such that $k\w=\v$.  
Since $\bar\gamma$ has the same endpoints as $\lambda$ and has no integer points in its interior, we can assume by symmetry that $\bar\gamma$ passes above the point $\w$.
Translating the graph of $\bar\gamma$ by $\w$, we obtain the graph of another strictly increasing function $\bar\gamma'$, and these two graphs are disjoint.
Thus since $\bar\gamma$ passes above $\w$ and $\bar\gamma'$ contains $\w$, we see that $\bar\gamma'$ passes above $2\w$.
Thus $\bar\gamma$ passes above $2\w$, and thus $\bar\gamma'$ passes above $3\w$.
Continuing in this manner, we conclude that the graph $\bar\gamma$ passes above $k\w$, contradicting the fact that $\v=k\w$ is on the curve $\bar\gamma$.

Now suppose there is an integer point $\w$ between the two graphs.
Among integer points between the two graphs, choose a point $\w$ that maximizes the distance from $\w$ to the line $\lambda$.
By symmetry, we may as well assume that $\w$ is above $\lambda$.
Since $\w$ is between the line $\lambda$ and the graph of $\bar\gamma$, we see that $\bar\gamma$ is above $\bar\gamma'$ (for $\bar\gamma'$ as in the previous paragraph).
Thus $\bar\gamma'$ is above the point $\w+\w$, so $\bar\gamma$ is above $\w+\w$, but the point $\w+\w$ is twice as far from $\lambda$ as the point $\w$.

In either case, we have a contradiction, and we conclude that the two words must be the same.
Thus $\gamma$ is isotopic to the projection of $\lambda$.
Furthermore, the word for $\lambda$ avoids the letter $\c$, so the line segment $\lambda$ has no integer points in its interior.
\end{proof}

The arcs constructed in Proposition~\ref{torus arcs} are indexed by rational slopes, including the infinite slope.
To specify these slopes, and thus these arcs, we introduce some (mostly standard) terminology.
A \newword{Farey point} is an integer vector $[a,\,b]$ that is a closest nonzero integer point to the origin on the line through the origin with slope $b/a$.
Taking the appropriate conventions for the $\gcd$ of not-necesarily-positive integers, this is the same as requiring that $\gcd(a,b)=1$.
A \newword{standard Farey point} satisfies the additional conditions that $a\ge0$ and that $b=1$ whenever $a=0$.
For each rational slope $q$ (including $\infty$), there are exactly two Farey points [a,\,b] such that $b/a=q$.
The \newword{standard form} of a (possibly infinite) rational slope is an expression $b/a$ such that $[a,\,b]$ is a standard Farey point.

The \newword{Farey ray} for a Farey point $[a,\,b]$ is the set $\set{c[a,\, b]:c\ge1}$.
\newword{Farey neighbors} are pairs of Farey points $[a,\,b]$ and $[c,\,d]$ satisfying the following three requirements: $a$ and $c$ have weakly the same sign, $b$ and $d$ have weakly the same sign, and $ad-bc=\pm1$.
A \newword{Farey triangle} is the convex hull of three Farey points which are pairwise Farey neighbors.
The following lemma justifies the name Farey triangle.
\begin{lemma}\label{Farey tri nondegenerate}
If three Farey points are pairwise Farey neighbors, then they are not collinear.
\end{lemma}
\begin{proof}
If Farey points $[a,\,b]$, $[c,\,d]$ and $[e,\,f]$ are collinear, then there exists a scalar $\lambda$ such that $[e,\,f]=\lambda[a,\,b]+(1-\lambda)[c,\,d]$.
The equations $af-be=\pm1$ and $cf-de=\pm1$ are constraints on $\lambda$.
Requiring also that $ad-bc=\pm1$, the constraints become $1-\lambda=\pm1$ and $\lambda=\pm1$, and this is a contradiction.
\end{proof}

Figure~\ref{Farey fig} shows the Farey points, triangles, and rays close to the origin.
\begin{figure}[ht]
\scalebox{1}{\includegraphics{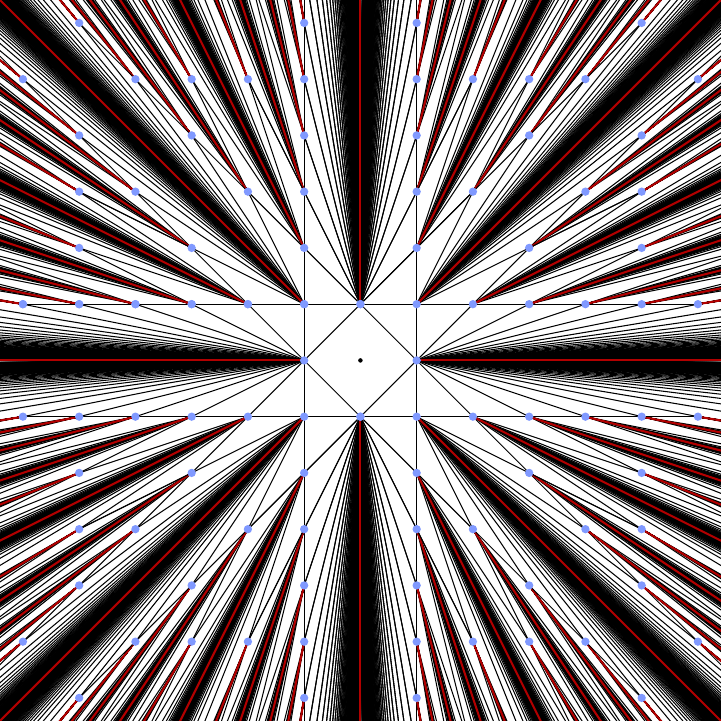}}
\caption{Farey points, triangles, and rays}
\label{Farey fig}
\end{figure}
Farey points are shown in blue and Farey rays are shown in red.
The black dot marks the origin.
(All lines in the figure are straight.  The apparent curvature of some lines is due to a well-known optical illusion.
See~\cite{HeringWiki}.)

Our interest in Farey neighbors stems from the following proposition.
\begin{prop}\label{torus arcs compatible}
Given two distinct arcs in $(\S,p)$, write the slopes of the corresponding line segments in standard forms $b/a$ and $d/c$ with $\frac ba<\frac dc$.
Then the arcs are compatible if and only if $ad-bc=1$.
\end{prop}
\begin{proof}
The condition $\frac ba<\frac dc$ is equivalent to $ad-bc>0$.
Thus it is enough to show that for $b/a\neq d/c$, the condition $ad-bc=\pm1$ is equivalent to the following condition:
Given a line with slope $b/a$ passing through an integer point and a line with slope $d/c$ passing through another integer point, the intersection of the two lines is an integer point.

Given a line with slope $b/a$ passing through the integer point $[m_1,\,p_1]$ and a line with slope $d/c$ passing through the integer point $[m_2,\,p_2]$, the intersection of the two lines is the integer point $[x,\,y]$ with 
\[x=\frac{adm_2-bcm_1+ac(p_1-p_2)}{ad-bc},\quad y=\frac{adp_1-bcp_2+bd(m_2-m_1)}{ad-bc}.\]
This is an integer point if $ad-bc=\pm1$.

On the other hand, suppose $ad-bc\neq\pm1$ and suppose for the sake of contradiction that all of the intersections of lines are integers.
Choose $m_1=1$ and $m_2=p_1=p_2=0$ so that $x=\frac{bc}{bc-ad}$ and $y=\frac{bd}{bc-ad}$ are both integers.
Then $bc$ and $bd$ are both divisible by $ad-bc$, but since $c$ and $d$ have no common factor, we conclude that $b$ is a multiple of $ad-bc$.
Choosing $p_1=1$ and $m_1=m_2=p_2=0$, we see that $a$ is also a multiple of $ad-bc$.
This is a contradiction because $b/a$ is in standard form.
\end{proof}

In light of Proposition~\ref{torus arcs compatible}, if we identify antipodal Farey points and Farey triangles in Figure~\ref{Farey fig}, we obtain a representation of the arc complex of the punctured torus.
A representation of the arc complex appears later in Figure~\ref{markov g fig}.

Having classified arcs and determined the arc complex, we are able to describe the allowable curves in $(\S,p)$.
First we find the allowable closed curves.

\begin{prop}\label{torus closed}
The allowable closed curves in $(\S,p)$ are the images, under the covering map, of straight lines in $\reals^2\setminus\integers^2$ with rational (or infinite) slope.
Two such images are isotopic if and only if the corresponding lines have the same slope.
\end{prop}
\begin{proof}
We first observe that for each arc $\alpha$ in $(\S,p)$, there is a unique closed allowable curve $\lambda$ in $(\S,p)$ such that $\alpha$ and $\lambda$ do not intersect, and this correspondence is bijective.
Indeed, given an arc $\alpha$, the set $\S\setminus\alpha$ is a cylinder and admits exactly one non-contractible simple closed curve up to isotopy.
Similarly, given a non-contractible simple closed curve $\lambda$, the set $\S\setminus\lambda$ is a cylinder containing $p$.
Thus there is exactly one arc contained in $\S\setminus\lambda$.
If $\alpha$ is the image, under the projection map, of a line segment as in Proposition~\ref{torus arcs}, the unique allowable closed curve not intersecting $\alpha$ is the projection of a line with the same slope.
The present proposition thus follows from Proposition~\ref{torus arcs}.
\end{proof}

Recall that tagged arcs in $(\S,p)$ are obtained from ordinary arcs by either tagging both ends plain or tagging both ends notched.
Given a tagged arc $\alpha$ in $(\S,p)$, we define (up to isotopy) a curve $\kappa(\alpha)$, which coincides with $\alpha$ except in a small ball about $p$.
If $\alpha$ is tagged plain at $p$, then $\kappa(\alpha)$ spirals clockwise into $p$ at both ends, and if $\alpha$ is tagged notched at $p$, then $\kappa(\alpha)$ spirals counterclockwise into $p$ at both ends.
This is a special case of a map defined in \cite[Section~5]{unisurface} from tagged arcs to allowable curves in any marked surface $(\S,\M)$, and \cite[Lemma~5.1]{unisurface} states that $\kappa$ is a bijection from tagged arcs to allowable curves that are not closed and that two tagged arcs $\alpha$ and $\gamma$ are compatible if and only if the curves $\kappa(\alpha)$ and $\kappa(\gamma)$ are compatible.
Thus Propositions~\ref{torus arcs} and~\ref{torus closed} imply the following statement.

\begin{prop}\label{torus allowable}
The allowable curves in $(\S,p)$ are as follows:
\begin{enumerate}
\item 
All curves obtained from the arcs described in Proposition~\ref{torus arcs} by replacing both endpoints by clockwise spirals into $p$.
\item 
All curves obtained from the arcs described in Proposition~\ref{torus arcs} by replacing both endpoints by counterclockwise spirals into $p$.
\item 
All closed curves described in Proposition~\ref{torus closed}.
\end{enumerate}
\end{prop}

We write $\cl(a,b)$ for the \emph{closed} curve that is the projection of a line with slope in standard form $b/a$.
For the other allowable curves associated to $b/a$, we write $\cw(a,b)$ for the curve with spirals \emph{clockwise} into $p$ and $\ccw(a,b)$ for the curve with spirals \emph{counterclockwise} into $p$.

\begin{prop}\label{torus compatible}
The compatible pairs of allowable curves in $(\S,p)$ are:
\begin{enumerate}
\item $\cl(a,b)$ and $\cw(a,b)$ for any slope $b/a$. 
\item $\cl(a,b)$ and $\ccw(a,b)$ for any slope $b/a$. 
\item $\cw(a,b)$ and $\cw(c,d)$ for slopes $b/a$ and $d/c$ with $ad-bc=\pm1$.
\item $\ccw(a,b)$ and $\ccw(c,d)$ for slopes $b/a$ and $d/c$ with $ad-bc=\pm1$.
\end{enumerate}
\end{prop}
\begin{proof}
Closed curves are compatible with other curves if and only if the corresponding slopes coincide.
The remainder of the proposition is a direct consequence of Proposition~\ref{torus arcs compatible} and the fact that two tagged arcs $\alpha$ and $\gamma$ are compatible if and only if the curves $\kappa(\alpha)$ and $\kappa(\gamma)$ are compatible.
\end{proof}

\section{Explicit shear coordinates}\label{expl shear}
In this section, we explicitly compute shear coordinates of allowable curves in the once-punctured torus.
The purpose of this explicit calculation is twofold. 
First, we will use the explicit coordinates to show that the once-punctured torus has the Null Tangle Property.
Second, the explicit coordinates appear in Theorem~\ref{markov main}.

For the remainder of the paper, let $T_0$ be the triangulation of $(\S,p)$ by three allowable arcs $\gamma_1$, $\gamma_2$, and $\gamma_3$, corresponding respectively to slopes $0$, $\infty$, and $-1$.
The signed adjacency matrix $B(T_0)$ is $B=\begin{bsmallmatrix*}[r]0&2&-2\\-2&0&2\\2&-2&0\\\end{bsmallmatrix*}$.  
The lifts of $\gamma_1$ to $\reals^2$ are the segments connecting integer points $[m,\,n]$ to adjacent points $[m+1,\,n]$.
Similarly, lifts of $\gamma_2$ connect $[m,\,n]$ to $[m,\,n+1]$ and lifts of $\gamma_3$ connect $[m,\,n]$ to $[m+1,\,n-1]$.

\begin{prop}\label{explicit coords}
The shear coordinates with respect to $T_0$ of allowable curves in $(\S,p)$~are as follows:
\begin{enumerate}
\item \label{gs}
The cyclic permutations of $[1-b,\,a+1,\,b-a-1]$ for Farey points $[a,\,b]$ with $a\ge 0$ and $b>0$.
(These are for curves with counterclockwise spirals.) 
\item \label{opp gs}
The cyclic permutations of $[-1-b,\,a-1,\,b-a+1]$ for Farey points $[a,\,b]$ with $a>0$ and $b\ge0$.
(These are for curves with clockwise spirals.) 
\item \label{equatorial}
The nonzero integer vectors $[x,\,y,\,z]$ with $x+y+z=0$ such that $x$, $y$, and $z$ have no common factors.
(These are for closed curves.)
\end{enumerate}
\end{prop}

\begin{proof}
We first consider the curves $\ccw(a,b)$ and $\cl(a,b)$ for positive or infinite slopes with standard form $b/a$.
As in the proof of Proposition~\ref{torus arcs}, each lifted curve determines a word in the letters $\r$ and $\t$, but we modify the words slightly to deal with spirals.
For $\ccw(a,b)$, we begin the word with $\t\r$, encoding the interaction of the curve with the square with corners $[0,\,0]$ and $[-1,\,1]$.
We then record $\r$ each time the lifted curve exits a square to the right and $\t$ each time the curve exits a square on top.
At the ending spiral point, we record an $\r$ as the curve exits a square on the right to begin its spiral, and then we end the word.
(More symmetrically, one might want to record $\r\t$ at the spiral, but the argument below shows that $\r$ is correct.)
For example, $\ccw(2,3)$ gives the word $\t\r\t\r\t\r$ as illustrated in the left picture of Figure~\ref{curve word}.
\begin{figure}[ht]
\begin{tabular}{cccc}
\includegraphics{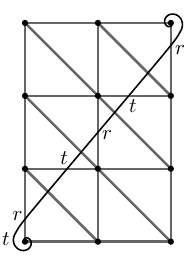}&&&\includegraphics{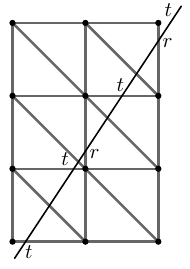}\\
$\t\r\t\r\t\r$&&&$\t\t\r\t\r\t$\\
$[-2,\,3,\,0]$&&&$[-3,\,2,\,1]$
\end{tabular}
\caption{Words and shear coordinates of curves $\ccw(3,2)$ and $\cl(3,2)$}
\label{curve word}
\end{figure}
If $b/a=\infty$, so that $a=0$ and $b=1$, then the word is $\t\r$.
For $\cl(a,b)$, we start at a point where some lift of the curve intersects the segment from $[0,\,0]$ to $[1,\,0]$, record a $\t$ for exiting the square with corners $[0,\,0]$ and $[1,\,-1]$ and record $\r$ and $\t$ until we reach a translate of the starting point, where we again record a $\t$.
Thus for example $\cl(2,3)$ gives the word $\t\t\r\t\r\t$ as illustrated in the right picture of Figure~\ref{curve word}.

To calculate shear coordinates from these words, we record a contribution for each consecutive pair of letters.
(Each interior letter in the word appear in two consecutive pairs.)
Figure~\ref{consec pair shear} illustrates these contributions.
(Cf. Figure~\ref{shear fig}.)
\begin{figure}[ht]
\begin{tabular}{cccc}
\includegraphics{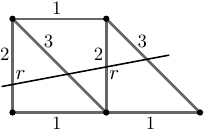}\,&\,\includegraphics{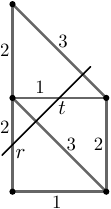}\,&\,\includegraphics{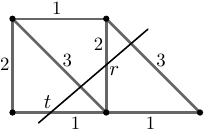}\,&\,\includegraphics{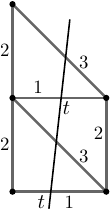}\\
$\r\r$&$\r\t$&$\t\r$&$\t\t$\\
$[0,\,1,\,-1]$&$[-1,\,0,\,0]$&$[0,\,1,\,0]$&$[-1,\,0,\,1]$
\end{tabular}
\caption{Shear coordinates from consecutive pairs in words}
\label{consec pair shear}
\end{figure}
In the figure, lifts of $\gamma_1$ are labeled ``$1$'' and so forth.
If the pair is $\r\r$, then we record the vector $[0,\,1,\,-1]$.
The $1$ in the $2\nd$ coordinate arises because after the second $\r$ is recorded, the curve next crosses a lift of $\gamma_3$.
If the pair is $\r\t$, then we record $[-1,\,0,\,0]$.
The $-1$ arises because after the $\t$ is recorded, the curve next crosses a lift of $\gamma_3$; we ended the word for $\ccw(a,b)$ with $\r$ instead of $\r\t$ to avoid incorrectly recording this $-1$.
If the pair is $\t\r$, then we record $[0,\,1,\,0]$.
If the pair is $\t\t$, then we record $[-1,\,0,\,1]$.
Adding up all of these contributions, we obtain the shear coordinates of the curve with respect to $T_0$.
Thus for example the curves represented in Figure~\ref{curve word} have shear coordinates as shown in the figure.

The word for $\ccw(a,b)$ is $\t\r$ followed by some sequence of $a-1$ instances of $\r$ and $b-1$ instances of $\t$, and then a final $\r$ for a total length of $a+b+1$ letters.
If $b/a\le1$, then the word has no consecutive pairs $\t\t$.
In this case, there are $b-1$ consecutive pairs $\r\t$ and $b$ consecutive pairs $\t\r$.  
The remaining $a-b+1$ consecutive pairs are $\r\r$.
Thus the shear coordinates of $\ccw(a,b)$ are 
\[(b-1)[-1,\,0,\,0]+b[0,\,1,\,0]+(a-b+1)[0,\,1,\,-1]=[1-b,\,a+1,\,b-a-1].\]
If $b/a>1$, then the word has no consecutive pairs $\r\r$.
(The condition $b/a>1$ implies in particular that the first letter after the initial $\t\r$ is $\t$ and that the last letter before the final $\r$ is also $\t$.)
Thus in this case, there are $a$ consecutive pairs $\r\t$ and $a+1$ consecutive pairs $\t\r$, with the remaining $b-a-1$ pairs being $\t\t$.
Thus the shear coordinates are 
\[a[-1,\,0,\,0]+(a+1)[0,\,1,\,0]+(b-a-1)[-1,\,0,\,1]=[1-b,\,a+1,\,b-a-1].\]

The calculation for $\cl(a,b)$ is very similar.
The word for $\cl(a,b)$ begins and ends with $\t$ and has a total of $a$ instances of $\r$ and $b+1$ instances of $\t$.
If $b/a\le1$, then the word has $\r\t$ $b$ times, $\t\r$ $b$ times, and $\r\r$ $a-b$ times, giving 
shear coordinates
\[b[-1,\,0,\,0]+b[0,\,1,\,0]+(a-b)[0,\,1,\,-1]=[-b,\,a,\,b-a].\]
If $b/a\ge 1$, then the word has $\r\t$ $a$ times, $\t\r$ $a$ times, and $\t\t$ $b-a$ times, giving the same formula for shear coordinates.
%\[a[-1,\,0,\,0]+a[0,\,1,\,0]+(b-a)[-1,\,0,\,1]=[-b,\,a,\,b-a].\]

We have thus dealt with shear coordinates of $\ccw(a,b)$ and $\cl(a,b)$ for $b/a$ positive or infinite.
The remaining possibilities for $b/a$ are now easily completed using the symmetries of the problem.
The linear map $[a,\,b]\mapsto[-b,\,a+b]$ restricts to a bijection from $\set{[a,\,b]:\,-a<b\le0}$ to $\set{[a,\,b]:\,a\ge0,\,b>0}$.
The map restricts further to a bijection on Farey points in these sets.
Furthermore, if $-a<b\le0$ and the shear coordinates of $\ccw(-b,a+b)$ are $[x,\,y,\,z]$, then by symmetry, the shear coordinates of $\ccw(a,\,b)$ are $[y,\,z,\,x]$.
The same is true replacing $\ccw(a,b)$ with $\cl(a,b)$.
Similarly, the map $[a,\,b]\mapsto[-a-b,\,a]$ restricts to a bijection from $\set{[a,\,b]:\,0<a\le-b}$ to $\set{[a,\,b]:\,a\ge0,\,b>0}$, and restricts further to a bijection on Farey points.
If $0<a\le-b$ and the shear coordinates of $\ccw(-a-b,a)$ are $[x,\,y,\,z]$, then the shear coordinates of $\ccw(a,\,b)$ are $[z,\,x,\,y]$, and the same is true for $\cl(a,b)$.

We have seen that the vectors described in item (1) of the theorem are exactly the shear coordinates of curves with counterclockwise spirals.
A reflection through the line $a=b$ induces a bijection from these curves to curves with clockwise spirals.
The effect of this reflection on shear coordinates is to switch $a$ and $b$, to swap the first two entries of the shear coordinates (since the reflection maps $\gamma_1$ to $\gamma_2$ and fixes $\gamma_3$) and to negate the shear coordinates (since the reflection reverses orientation).
Thus the vectors described in item (2) of the theorem are exactly the shear coordinates of curves with clockwise spirals.

We have also seen that the shear coordinates of closed curves are given by cyclic permutations of $[-b,\,a,\,b-a]$ for $b/a$ the standard form for a positive rational number or $\infty$.
To complete the proof, it remains only to show that all nonzero integer vectors $[x,\,y,\,z]$ with $x+y+z=0$ such that $x$, $y$, and $z$ have no common factors appear as cyclic permutations of such vectors $[-b,\,a,\,b-a]$.
Any such vector $[x,\,y,\,z]$ has one or more of its coordinates negative and one or more of its coordinates positive.
In particular, up to a cyclic permutation of coordinates, we can take $x$ negative and $y$ nonnegative.
Then setting $b=-x$ and $a=y$, we obtain $z=b-a$, so that $[x,\,y,\,z]$ is $[-b,\,a,\,b-a]$.
\end{proof}

\section{The Null Tangle Property}\label{null tangle sec}
In this section, we prove Theorem~\ref{torus null tangle}.
The following proposition goes part of the way towards the proof.

\begin{prop}\label{null tangle lines only}
Suppose $\Xi$ is a null tangle in $(\S,p)$.
Then the support of $\Xi$ contains only closed curves.
\end{prop}

The proof of Proposition~\ref{null tangle lines only} uses the following lemma, which is far from an exhaustive list of Farey neighbors.

\begin{lemma}\label{Farey neighbor}
Let $[a,\,b]$ be a standard Farey point.
\begin{enumerate}
\item \label{Farey neighbor a=0}
If $a=0$, then the Farey neighbors of $[a,\,b]$ include $[1,\,n]$ for $n\ge 0$. 
\item If $a=1$, then the Farey neighbors of $[a,\,b]$ include $[n,\,bn-1]$ for $n\ge 1$.
\item If $a>1$, then there exists an integer $c_0$ with $1\le c_0<a$ such that the Farey neighbors of $[a,\,b]$ include $\bigl[(c_0+an),\,\frac{b(c_0+an)-1}{a}\bigr]$ for $n\ge 0$.
\end{enumerate}
\end{lemma}
\begin{proof}
The condition $ad-bc=\pm1$ implies that $[c,\,d]$ is a Farey point because any common divisor of $c$ and $d$ is a divisor of $ad-bc$.
Thus, the Farey neighbors of $[a,\,b]$ are the integer points $[c,\,d]$ such that $a$ and $c$ have weakly the same sign, $b$ and $d$ have weakly the same sign, and $ad-bc=\pm1$.

If $a=0$, then $b=1$, and the statement is easy.
If $a=1$, then the statement is also easy.
Now suppose $a>1$.
For each integer $c\ge 1$, the equation $ad-bc=-1$ has an integer solution $d=\frac{bc-1}{a}$ if and only if $bc\equiv1\mod a$.
Since $\gcd(a,b)=1$, every integer $bc$ for $1\le c<a$ is in a distinct class mod $a$ and none of these integers are multiples of $a$.
Thus there exists an integer $c_0$ with $1\le c_0<a$ and $bc_0\equiv1\mod a$.
For each integer $n\ge0$, the quantity $d=\frac{b(c_0+an)-1}{a}$ is also an integer.
\end{proof}

\begin{proof}[Proof of Proposition~\ref{null tangle lines only}]
Suppose $\Xi$ is a null tangle in $(\S,p)$.
Proposition~\ref{torus allowable} says that each curve in $\Xi$ is $\cw(a,b)$, $\ccw(a,b)$, or $\cl(a,b)$ for some slope with standard form $b/a$.
For each curve $\lambda$, we write $q_\lambda$ for this slope $b/a$.

First, suppose $\Xi$ contains some curve $\lambda$ of the form $\ccw(e,f)$ for some slope $q_\lambda$ with standard form $f/e$.
Since there are finitely many curves in $\Xi$, there is some slope $\bar q<q_\lambda$ such that every curve $\nu\in\Xi$ has either $q_\nu\ge q_\lambda$ or $q_\nu\le\bar q$.
Lemma~\ref{Farey neighbor} lets us find a Farey neighbor $[a,\,b]$ of $[e,\,f]$ with $\bar q<b/a<f/e$.
This is done by taking $n$ large enough and setting $[a,\,b]$ equal to $[1,\,n]$ if $e=0$, or $[n,\,fn-1]$ if $e=1$, or $\bigl[(a_0+en),\,\frac{f(a_0+en)-1}{e}\bigr]$ if $e>1$.
Let $[c,\,d]$ be the Farey neighbor of $[e,\,f]$ obtained by replacing $n$ by $n+1$ in the expression for $[a,\,b]$.
Then $\bar q<b/a<d/c<f/e$, and furthermore, $[a,\,b]$ and $[c,\,d]$ are Farey neighbors.
Let $\alpha$ be the arc associated to $[a,\,b]$, let $\beta$ be the arc associated to $[c,\,d]$ and let $\gamma$ be the arc associated to $[e,\,f]$.
Then $T=\set{\alpha,\beta,\gamma}$ is a triangulation of $(\S,p)$.

To calculate shear coordinates, we look at the intersections of lifts of the curves with lifts of the arcs $\alpha,\beta,\gamma$.
When we do so, we see that $b_\gamma(T,\lambda)=1$, as illustrated in Figure~\ref{nulltangletorusfig}, with two preimages of $\alpha$ drawn in red, two preimages of $\beta$ drawn in blue, and a preimage of $\gamma$ drawn in purple.
\begin{figure}[ht]
\includegraphics{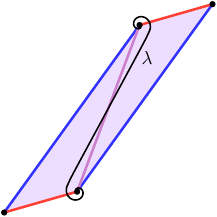}
\caption{Illustration for the proof of Proposition~\ref{null tangle lines only}}
\label{nulltangletorusfig}
\end{figure}
Furthermore, any other curve $\nu$ with $b_\gamma(T,\nu)>0$ has $b/a<q_\nu<f/e$.
But no such curve $\nu$ exists in $\Xi$ because $\bar q<b/a<f/e=q_\lambda$.
Thus any $\nu\in\Xi$ with $\nu\neq\lambda$ has $b_\gamma(T,\nu)$ nonpositive.
Proposition~\ref{one positive or one negative} says that $\lambda$ has weight $0$ in~$\Xi$.

Next suppose $\Xi$ contains some curve $\lambda$ of the form $\cw(a,b)$.
Applying the symmetry of the once-punctured torus induced by reflecting $\reals^2$ in the line $a=b$, we obtain a null tangle $\Xi'$.
Let $\lambda'$ be the image of $\lambda$ under the symmetry.
By the preceding argument, the weight of $\lambda'$ in $\Xi'$ is $0$, so the weight of $\lambda$ in $\Xi$ is~$0$.
\end{proof}

The following proposition is the specialization of \cite[Proposition~7.9]{unisurface} %This is Proposition {null tangle is B coher rel} of {unisurface}
to the once-punctured torus.

\begin{prop}\label{null tangle is B coher rel}
A tangle $\Xi$ in $(\S,\M)$ is null if and only if $\sum_{\lambda\in\Xi}w_\lambda\b(T,\lambda)$ is a $B(T)$-coherent linear relation for any triangulation $T$.
\end{prop}

\begin{proof}[Proof of Theorem~\ref{torus null tangle}]
Suppose $\Xi$ is a null tangle.
By Proposition~\ref{null tangle lines only}, the support of $\Xi$ contains only curves of the form $\cl(a,b)$.
To complete the proof, we resort to working with explicit shear coordinates and Proposition~\ref{null tangle is B coher rel}.
Let $\lambda=\cl(a,b)$ be a curve in~$\Xi$ with weight $w_\lambda$.
Let $\gamma$ be the arc that is the projection of the segment connecting $[0,\,0]$ to $[a,\,b]$.
Complete $\gamma$ to a triangulation $T'$ by adding two additional arcs.
Index the arcs so that the shear coordinates of $\lambda$ are $[1,\,-1,\,0]$.
If $B(T')=-B$, then we alter $T'$ by flipping the arc $\gamma$.
%SinceSLC:  Changed $\gamma$ to $\lambda$ in following sentence.  Thanks, Emily! 
This preserves the shear coordinates of $\lambda$ and makes $B(T')$ equal to $B$.
Taking the shear coordinates, with respect to $T'$ of all of the curves in $\Xi$ and applying Proposition~\ref{null tangle is B coher rel}, we obtain a $B$-coherent linear relation $c\v+\sum_{i\in S}c_i\v_i$ on nonzero integer vectors in the plane $x+y+z=0$, with $\v=[1,\,-1,\,0]$ and $c=w_\lambda$.
We will show that $c=0$.

To show that $c=0$, we consider $\eta_{21}^B=\eta^{-B}_2\circ\eta^B_1$ and its inverse $\eta_{12}^B$, restricting our attention to the plane $x+y+z=0$.
Using \eqref{mutation map def}, we calculate $\eta_{21}^B([x,\,y,\,z])$ to be
\[
\eta_{21}^B([x,\,y,\,z])=
\begin{cases}
[-x,\,-y,\,2x+2y+z]&\text{if $x\le0$ and $y\le0$},\\
[-x+2y,\,-y,\,2x+z]&\text{if $x\le0$ and $y\ge0$},\\
[-x,\,-2x-y,\,4x+2y+z]&\text{if $x\ge0$ and $2x+y\le0$},\\
[3x+2y,\,-2x-y,\,z]&\text{if $x\ge0$ and $2x+y\ge0$}.\\
\end{cases}
\]
Define $D_0$ to be the closed cone spanned by $[-1,\,0,\,1]$ and $[0,\,1,\,-1]$.
For each integer $j>0$, define $D_j$ to be the closed cone spanned by $[j-1,\,-j+2,\,-1]$ and $[j,\,-j+1,\,-1]$.
For each integer $j<0$, define $D_j$ to be the closed cone spanned by $[-j-2,\,j+1,\,1]$ and $[-j-1,\,j,\,1]$.
Some of these cones are illustrated in Figure~\ref{coords and eta12}.
\begin{figure}[h]
\begin{tabular}{cccc}
\scalebox{1}{\includegraphics{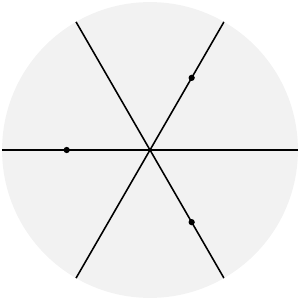}}
\begin{picture}(0,0)(72,-72)
\put(-43,-52){\rotatebox{60}{$x=0$}}
%\put(-33,-58){\rotatebox{60}{$y+z=0$}}
\put(31,5){$y=0$}
%\put(13.25,-8){$x+z=0$}
\put(-33,57){\rotatebox{-60}{$z=0$}}
%\put(-44,51){\rotatebox{-60}{$x+y=0$}}
\put(19,29){\small$[0,\,1,\,-1]$}
\put(19,-33){\small$[1,\,-1,\,0]$}
\put(-63,5){\small$[-1,\,0,\,1]$}
\end{picture}
&&&
\scalebox{1}{\includegraphics{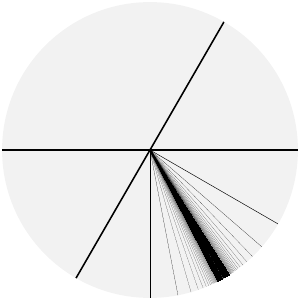}}
\begin{picture}(0,0)(72,-72)
\put(-30,30){$D_0$}
\put(30,20){$D_1$}
\put(38,-15){$D_2$}
\put(-50,-26){$D_{-1}$}
\put(-27,-55){$D_{-2}$}
\end{picture}
\end{tabular}
\caption{The plane $x+y+z=0$ and the cones $D_j$}
\label{coords and eta12}
\end{figure}
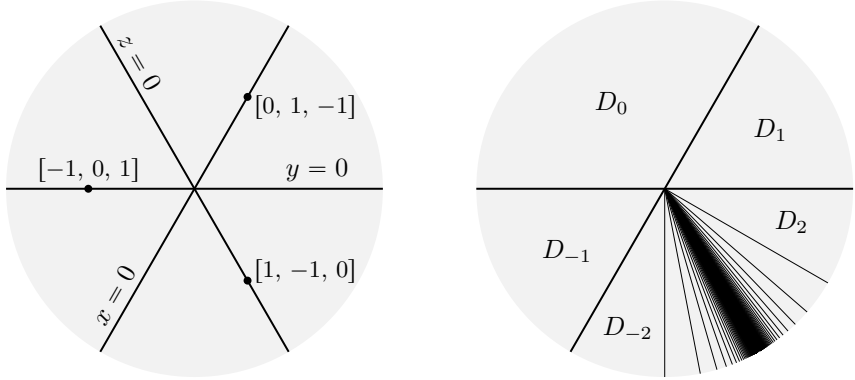
The left picture establishes the coordinate system for the right picture by showing the intersections of the planes $x=0$, $y=0$ and $z=0$ with the plane $x+y+z=0$ and by labeling several points in the plane.
The map $\eta_{21}^B$ fixes $\v$ and sends $D_j$ to $D_{j+2}$ for all $j$.
The inverse map $\eta_{12}^B$ fixes $\v$ and sends $D_j$ to $D_{j-2}$ for all $j$.

Now, we may as well assume that all of the vectors $\v_i$ are contained in $D_1\cup\cdots\cup D_{m-2}$ for some $m\ge 3$.
Otherwise, we can replace $c\v+\sum_{i\in S}c_i\v_i$ by $c\v+\sum_{i\in S}c_i(\eta_{21}^B)^\ell\v_i$ for large enough integer $\ell$.
Partition the set $S$ into $m$ (possibly empty) blocks $U_1,\ldots,U_m$ with the following property:
If $i\in U_j$, then $\v_i\in D_j$.
There may not be a unique way to do this, because $\v_i$ may be in $U_j\cap U_{j+1}$.
For each $j$ from $1$ to $m$, let $\w_j=\sum_{i\in U_j}c_i\v_i$.
Since this sum may have more than one term and since the $c_i$ may be negative, the vector $\w_j$ may not be contained in $D_j$.
For that reason, $\eta_{12}^B(\w_j)$ may not equal $\sum_{i\in U_j}c_i\eta_{12}^B(\v_i)$.
Each $D_j$ is contained in a domain of linearity of $\eta_{12}^B$, and we define $\theta_j$ to be the linear map that agrees with $\eta_{12}^B$ on $D_j$.
Furthermore, each $D_j$ is contained in a domain of linearity of $(\eta_{12}^B)^k$ for any $k$.
We write $\theta^{(k)}_j$ for the linear map that agrees with $(\eta_{12}^B)^k$ on $D_j$.
Thus $\theta^{(k)}_j$ is $\theta_{j+2k-2}\circ\theta_{j+2k-4}\circ\cdots\circ\theta_j$ for $k\ge 1$ or $\theta_{j+2k+2}^{-1}\circ\theta_{j+2k+4}^{-1}\circ\cdots\circ\theta_j^{-1}$ for $k\le1$.
A key to the following argument is that, for all $j$, the maps $\theta_j$ and $\theta_{j+1}$ coincide on the line containing $D_j\cap D_{j+1}$, and both maps send that line to the line containing $D_{j+2}\cap D_{j+3}$.
%Since each $\theta_j$ maps $D_j$ to $D_{j-2}$, the maps $\theta_j^{-1}$ and $\theta_{j+1}^{-1}$ coincide on the line containing $D_{j-2}$ and $D_{j-1}$.

Now we begin to apply the fact that $c\v+\sum_{i\in S}c_i\v_i$ is a $B$-coherent linear relation.
Specifically, we inductively construct a list $\u_0,\u_1,\ldots,\u_m$ of vectors with two properties:
First, each $\u_\ell$ is in the line containing the ray $D_\ell\cap D_{\ell+1}$, and second, $\w_j=\u_j-\u_{j-1}$ for $j$ from $1$ to $m$.
We begin by setting $\u_0=\mathbf{0}$.
As a base for the induction, observe that all of the vectors $\v_i$ with positive second coordinate have $i\in U_1$.
Thus, subtracting \eqref{piecewise eta} from \eqref{linear eta} with $\k=\emptyset$ and restricting to the second coordinate, we see that $\w_1$ equals some vector $\u_1$ contained in the line $y=0$, which is the line containing $D_1\cap D_2$.
Now suppose $j>1$.
By induction, we have already constructed a vector $\u_{j-2}$ contained in the line containing $D_{j-2}\cap D_{j-1}$ and a vector $\u_{j-1}$ contained in the line containing $D_{j-1}\cap D_j$, and we have $\w_{j-1}=\u_{j-1}-\u_{j-2}$.

Consider first the case where $j$ is even.
All of the vectors $(\eta_{12}^B)^{j/2}(\v_i)$ with negative first coordinate are in the interior of $D_{-1}\cup D_0$, so that each corresponding $\v_i$ is in the interior of $D_{j-1}\cup D_j$, or in other words, $i\in U_{j-1}\cup U_j$.
Restricting \eqref{piecewise eta} to its first coordinate, with $\k$ the sequence $1212\cdots$ of length $j$, we see that $\theta_{j-1}^{(j/2)}(\w_{j-1})+\theta_j^{(j/2)}(\w_j)$ is a vector contained in the line $x=0$.
Now $\theta_{j-1}^{(j/2)}(\w_{j-1})=\theta_{j-1}^{(j/2)}(\u_{j-1})-\theta_{j-1}^{(j/2)}(\u_{j-2})$.
Since $\u_{j-2}$ is in the line containing $D_{j-2}\cap D_{j-1}$ and $\u_{j-1}$ is in the line containing $D_{j-1}\cap D_j$, we see that $\theta_{j-1}^{(j/2)}(\u_{j-2})$ is in the line containing $D_{-2}\cap D_{-1}$ (i.e.\ the line $x=0$) and that $\theta_{j-1}^{(j/2)}(\u_{j-1})$ is the line containing $D_{-1}\cap D_0$ (i.e.\ the line $y=0$).
Now the fact that $\theta_{j-1}^{(j/2)}(\w_{j-1})+\theta_j^{(j/2)}(\w_j)$ is in the line $x=0$ means that 
$\theta_j^{(j/2)}(\w_j)=\s-\theta_{j-1}^{(j/2)}(\u_{j-1})$ for some vector $\s$ in the line $x=0$.
Since $\u_{j-1}$ is in the line containing $D_{j-1}\cap D_j$, the maps $\theta_{j-1}^{(j/2)}$ and $\theta_j^{(j/2)}$ agree on the vector $\u_{j-1}$.
Thus we write $\theta_j^{(j/2)}(\w_j)=\s-\theta_j^{(j/2)}(\u_{j-1})$, and we conclude that $\w_j=\u_j-\u_{j-1}$,
where $\u_j=(\theta_j^{(j/2-1)})^{-1}\s$ is contained in the line containing $D_j\cap D_{j+1}$.

Next, consider the case where $j$ is odd.
All of the vectors $(\eta_{12}^B)^{(j-1)/2}(\v_i)$ with positive second coordinate are in the interior of $D_0\cup D_1$, so that each corresponding $i$ is in $U_{j-1}\cup U_j$.
Thus, subtracting \eqref{piecewise eta} from \eqref{linear eta} with $\k$ the sequence $1212\cdots$ of length $j-1$ and restricting to the second coordinate, we see that $\theta_{j-1}^{((j-1)/2)}(\w_{j-1})+\theta_j^{((j-1)/2)}(\w_j)$ is a vector contained in the line $y=0$, which is the line containing $D_1\cap D_2$.
But $\theta_j^{((j-1)/2)}(\u_{j-2})$ is also in the line $y=0$ containing $D_{-1}\cap D_0$, so we conclude that $\theta_j^{((j-1)/2)}(\w_j)=\s-\theta_{j-1}^{((j-1)/2)}(\u_{j-1})$ for some vector $\s$ in the line $y=0$.
Since $\u_{j-1}$ is in the line containing $D_{j-1}\cap D_j$, the maps $\theta_{j-1}^{((j-1)/2)}$ and $\theta_j^{((j-1)/2)}$ agree on the vector $\u_{j-1}$.
Thus we write $\theta_j^{((j-1)/2)}(\w_j)=\s-\theta_j^{((j-1)/2)}(\u_{j-1})$, and we conclude that $\w_j=\u_j-\u_{j-1}$ where $\u_j=(\theta_j^{((j-1)/2)})^{-1}\s$ is contained in the line containing $D_j\cap D_{j+1}$
This completes our inductive construction of the $\u_j$.

By our choice of $m$, the set $U_m$ is empty, so $\w_m=\mathbf{0}$.
Since $\w_m$ is $\u_m-\u_{m-1}$ with $\u_m$ and $\u_{m-1}$ contained in distinct lines through the origin, we conclude that $\u_m=\u_{m-1}=\mathbf{0}$.
We rewrite the expression $c\v+\sum_{i\in S}c_i\v_i$ as $c\v+\sum_{j=1}^m\w_j$, and then further as a telescoping sum $c\v+\sum_{j=1}^m(\u_j-\u_{j-1})=c\v+\u_m-\u_0$, which equals $c\v$.
But this expression equals $\mathbf{0}$ and $\v$ is a nonzero vector, so $c=0$.
Since $\lambda$ is an arbitrary curve in $\Xi$, the proof is complete.
\end{proof}

Now Corollary~\ref{torus pos basis} combines with Proposition~\ref{explicit coords} to show that the vectors described in \eqref{markov gs}, \eqref{markov opp gs}, and \eqref{markov equatorial} of Theorem~\ref{markov main} are a $\integers$- or $\rationals$-basis for $B$.
By Theorem~\ref{basis univ}, we have proved the first assertion of Theorem~\ref{markov main}.
The proof of Theorem~\ref{markov main} is completed in Section~\ref{real sec}.

\section{The mutation fan}\label{fan sec}
In this section, we construct the mutation fan $\F_B$, which was pictured in Figure~\ref{torus FB fig}.
We define six maps from $\reals^2$ to $\reals^3$.
\[\phi_1:[a,\,b]\mapsto[1-b,\,a+1,\,b-a-1]\qquad\phi_4:[a,\,b]\mapsto[-1-b,\,a-1,\,b-a+1]\]
\[\phi_2:[a,\,b]\mapsto[a+1,\,b-a-1,\,1-b]\qquad\phi_5:[a,\,b]\mapsto[a-1,\,b-a+1,\,-1-b]\]
\[\phi_3:[a,\,b]\mapsto[b-a-1,\,1-b,\,a+1]\qquad\phi_6:[a,\,b]\mapsto[b-a+1,\,-1-b,\,a-1]\]
The first three maps are bijections from $\reals^2$ to the plane given by $x+y+z=1$, related to each other by a cyclic permutation of coordinates.
The last three maps are bijections to the plane $x+y+z=-1$, related in the same way.
We extend each map $\phi_i$ to a map $\Phi_i$ from closed polyhedra in $\reals^2$ to closed polyhedral cones in $\reals^3$ sending a polyhedron $P$ to the closure of the nonnegative linear span of $\phi_i(P)$.
In particular, these maps send Farey rays and Farey triangles to $2$- and $3$-dimensional simplicial cones.
A Farey ray with vertex $[a,\,b]$ is the limit, as $c\to\infty$, of the line segment from $[a,\, b]$ to $c[a,\, b]$.
Thus $\Phi_i$ sends the Farey ray to the limit, as $c\to\infty$, of the $2$-dimensional cone spanned by $\phi_i([a,\,b])$ and $\phi_i([ca,\,cb])$, or equivalently the cone spanned by $\phi_i([a,\,b])$ and $\frac1c\phi_i([ca,\,cb])$.
This limit is the cone spanned by $\phi_i([a,\,b])$ and $\lim_{c\to\infty}\frac{1}{c}\phi_i([ca,\,cb])$.
For example, $\Phi_1$ sends this ray to the cone spanned by $[1-b,\,a+1,\,b-a-1]$ and $[-b,\,a,\,b-a]$.

\begin{theorem}\label{torus FB}
Let $B=\begin{bsmallmatrix*}[r]0&2&-2\\-2&0&2\\2&-2&0\\\end{bsmallmatrix*}$.
Then $\F_B$ consists of the following cones and their faces:
\begin{enumerate}
\item \label{torus FB pmO}
The nonnegative and nonpositive orthants $\pm O$.
\item \label{torus FB triangle}
For each $i=1,2,3,4,5,6$, the image under $\Phi_i$ of all Farey triangles contained in the region $\set{[a,\,b]\in\reals^2:a\ge-1,\,b\ge1}$.
\item \label{torus FB ray}
For each $i=1,2,3,4,5,6$, the image under $\Phi_i$ of all Farey rays whose vertex is $[a,\,b]$ with $a\ge0$ and $b\ge 1$.
\item
All rays contained in the plane $x+y+z=0$.
\end{enumerate}
All of the cones described above are maximal except for the rational rays in the plane $x+y+z=0$, each of which is a proper face of two cones described in \eqref{torus FB ray}.
\end{theorem}

To prove Theorem~\ref{torus FB}, we need the following lemma.

\begin{lemma}\label{Farey decomp}
Every point in $\set{[a,\,b]\in\reals^2:b\ge1}$ is contained in some Farey triangle or ray.
\end{lemma}

\begin{proof}
Suppose $[a,\,b]\in\reals^2$ has $b\ge1$.
We argue by induction on $\lceil |a|\rceil$ (the smallest integer greater than or equal to $|a|$).
For $0\le\lceil a\rceil\le 1$, the lemma follows from Lemma~\ref{Farey neighbor}\eqref{Farey neighbor a=0}.
If $a>1$, then we apply the piecewise-linear map given by $[a,\,b]\mapsto[a-b,\,b]$ (if $a\ge b$) or $[a,\,b]\mapsto[b-a,\,a]$ (if $a\le b$).
This map preserves Farey points, triangles, and rays.
By induction, the image of $[a,\,b]$ under the map is contained in some Farey triangle or ray.
Therefore, the same is true of $[a,\,b]$.
If $a<0$, then we argue symmetrically and appeal to induction.
\end{proof}

\begin{proof}[Proof of Theorem~\ref{torus FB}]
We first explicitly construct $\F_\rationals(T_0)$.
The rays of $\F_\rationals(T_0)$ are spanned by the shear coordinates of allowable curves in $(\S,p)$.
Proposition~\ref{explicit coords} shows that these rays are the images of standard Farey points under the maps $\Phi_i$ and the rational rays in the plane $x+y+z=0$.
We have $\lim_{c\to\infty}\frac{1}{c}\phi_1([ca,\,cb])=[-b,\,a,\,b-a]$ for a Farey point $[a,\,b]$ with $a\ge0$ and $b\ge 1$.
The proof of Proposition~\ref{explicit coords} shows that $[-b,\,a,\,b-a]$ is $\b(T_0,\cl(a,b))$.
By Proposition~\ref{torus compatible}, we see that $\Phi_1$ takes the Farey ray with vertex $[a,\,b]$ to the $2$-dimensional cone of $\F_\rationals (T_0)$ spanned by the shear coordinates of $\cl(a,b)$ and $\ccw(a,b)$.
By symmetry, $\Phi_i$ maps this Farey ray to a $2$-dimensional cone of $\F_\rationals (T_0)$ for all $i$.

Proposition~\ref{torus compatible} also implies that the full-dimensional cones of $\F_\rationals(T_0)$ are obtained as follows:
Take exactly one representative of each antipodal pair of Farey triangles.
Each representative indexes two full-dimensional cones.
One is obtained by computing shear coordinates of $\ccw(\cdot,\cdot)$ of each vertex and taking nonnegative span of the three resulting vectors.
The other is obtained in the same way with $\cw(\cdot,\cdot)$.
We deal first with the cones obtained from curves with counterclockwise spirals.

The proof of Proposition~\ref{explicit coords} establishes that, for Farey triangles contained in the region $\set{[a,\,b]\in\reals^2:a\ge0,\,b>0}$, the maximal cones for curves with counterclockwise spirals are obtained by applying the map $\Phi_1$.
Furthermore, one checks directly that the shear coordinate vector of $\ccw([1,\,-n])$ is $[1-n,\,0,\,n]=\phi_1([-1,\,n])$.
Thus the map $\Phi_1$ takes all Farey triangles contained in $\set{[a,\,b]\in\reals^2:a\ge-1,\,b\ge1}$ to full-dimensional cones of $\F_\rationals(T_0)$.
The map $[a,\,b]\mapsto[b,\,a-b]$ takes the region $\set{[a,\,b]\in\reals^2:a+b\ge1,\,b\le1}$ to $\set{[a,\,b]\in\reals^2:a\ge-1,\,b\ge1}$ and, as pointed out in the proof of Proposition~\ref{explicit coords}, corresponds to a cyclic permutation of the shear coordinates of $\ccw(a,b)$.
The map $[a,\,b]\mapsto[a,\,-a-b]$ takes the region $\set{[a,\,b]\in\reals^2:a\ge1,\,a+b\le1}$ to $\set{[a,\,b]\in\reals^2:a\ge-1,\,b\ge1}$ and corresponds to the other cyclic permutation of shear coordinates of $\ccw(a,b)$.
The only remaining pair of antipodal Farey triangles is $[0,\,1]$, $[-1,\,0]$, $[-1,\,1]$ and its antipodal opposite.
The shear coordinates of $\ccw(a,\,b)$ for vertices $[a,\,b]$ of this triangle span the cone $O$.

We have established that the maximal cones in $\F_\rationals(T_0)$ spanned by curves with counterclockwise spirals are given by $O$ and the cones described in \eqref{torus FB triangle} for $i=1,2,3$.
To construct the maximal cones in $\F_\rationals(T_0)$ spanned by curves with clockwise spirals, we use the same symmetry we used in the proof of Proposition~\ref{explicit coords} (reflection through the line $a=b$).
Specifically, the cones for clockwise spirals are obtained from the cones for counterclockwise spirals by switching $a$ and $b$, swapping the first two coordinates, and negating the shear coordinates.
Thus the cones for clockwise spirals are $-O$ and the cones described in \eqref{torus FB triangle} for $i=4,5,6$.
We have now shown that the maximal cones in $\F_\rationals(T_0)$ are given by \eqref{torus FB pmO}, \eqref{torus FB triangle}, and~\eqref{torus FB ray}.
Each cone in \eqref{torus FB ray} has a rational ray in the plane $x+y+z=0$ as an extreme ray, and all rational cones in this plane occur, by Proposition~\ref{explicit coords}.

Let $U$ be the set $\set{[a,\,b]\in\reals^2:a\ge-1,\,b\ge1}$.
Then 
\[\phi_1(U)=\set{[x,\,y,\,z]\in\reals^3:x+y+z=1,\,x\le0,\,y\ge0}\] 
and $\phi_2(U)$ and $\phi_3(U)$ are described similarly.
We see that the plane $x+y+z=1$ is the union 
\[\phi_1(U)\cup\phi_2(U)\cup\phi_3(U)\cup\set{[x,\,y,\,x]\in\reals^3:x+y+z=1,\,x\ge0,\,y\ge0,\,z\ge0}.\]
By this observation, we now show that every point in $\reals^3$ outside of the plane $x+y+z=0$ is in some cone of $\F_\rationals(T_0)$.
Given such a point, some positive scaling of the point is in the plane $x+y+z=1$ or in the plane $x+y+z=-1$.
If the point scales to the plane $x+y+z=1$, then it is in $O$ or in $\Phi_1(U)$, $\Phi_2(U)$, or $\Phi_3(U)$.
Lemma~\ref{Farey decomp} implies that the point is in $O$ or in the image of some Farey triangle or ray under $\Phi_1$, $\Phi_2$, or $\Phi_3$.
Thus the point is in some cone of $\F_\rationals(T_0)$.
By symmetry, a point with a positive scaling to the plane $x+y+z=-1$ is also in some cone in $\F_\rationals(T_0)$.
We see that the images, under $\Phi_i$ for $i=1,\ldots,6$ of Farey triangles and rays in $U$ cover the set of real vectors outside of the plane $x+y+z=0$.
Furthermore, since the images of Farey rays are two-dimensional and there are countably many Farey rays, the images of Farey triangles in $U$ are dense in $\reals^3$.

Next, we show that all of the cones in $\F_\rationals(T_0)$ are cones in $\F_B$.
Proposition~\ref{rat FB summary} states that the full-dimensional cones in $\F_\rationals(T_0)$ are $B$-cones.
The image $C$ of a Farey ray under some map $\Phi_i$ is contained in a $B$-cone $D$ by Proposition~\ref{rat FB summary}.
The images of Farey triangles are disjoint from the plane $x+y+z=0$, except at the origin, but $D$ contains a nonzero point in the plane $x+y+z=0$.
Thus $D$ is not the image of a Farey triangle in $U$.
Since images of Farey triangles in $U$ are dense in $\reals^3$, we see that $D$ is $2$-dimensional.
Furthermore, the only possibilities for $D$ are that $D=C$ or that $D$ is the union of $C$ with the other cone in $\F_\rationals(T_0)$ (on the other side of the plane $x+y+z=0$) that shares an extreme ray with $C$.
We rule out the latter case, however, because in that case $D$ is a rational $B$-cone but not a cone in $\F_\rationals(T_0)$, contradicting Proposition~\ref{rat FB summary}.
Thus $C$ equals the $B$-cone $D$.

It remains only to show that every ray in the plane $x+y+z=0$ is a distinct $B$-cone.
Suppose to the contrary that two non-parallel vectors $[x,\,y,\,z]$ and $[x',\,y',\,z']$ in the plane $x+y+z=0$ are in the same $B$-cone $C$.
Find an integer point $[x'',\,y'']=c_1[x,\,y]+c_2[x',\,y']$ with $c_1>0$ and $c_2>0$ and set $z''=-x''-y''$.
Then $[x'',\,y'',\,z'']=c_1[x,\,y,\,z]+c_2[x',\,y',\,z']$ is in $C$ because $B$-cones are convex cones.
Also $[x'',\,y'',\,z'']$ is an integer vector in the plane $x+y+z=0$, so as shown above, it spans an extreme ray of a $2$-dimensional maximal $B$-cone $C'$ not contained in the plane $x+y+z=0$.
But $C\cap C'$ is not a face of $C$, contradicting Theorem~\ref{fan}.
\end{proof}

In \cite[Proposition~5.2]{unisurface}, the $\g$-vectors of cluster variables are realized in terms of shear coordinates of allowable curves, and the cones spanned by $\g$-vectors of clusters are realized as the nonnegative linear spans of sets of pairwise compatible allowable curves.
Combining this realization with Theorem~\ref{torus FB}, we obtain the following corollary.

\begin{cor}\label{markov g}
Let $B$ be the matrix of Theorem~\ref{torus FB}, so that $B^T=\begin{bsmallmatrix*}[r]0&-2&2\\2&0&-2\\-2&2&0\\\end{bsmallmatrix*}$.
Then the $\g$-vectors of cluster variables associated to $B^T$ are the cyclic permutations of vectors $[1-b,\,a+1,\,b-a-1]$, for all (possibly infinite) positive rational slopes with standard form $b/a$.
The cones spanned by $\g$-vectors of maximal sets of pairwise compatible cluster variables are exactly the images, under the maps $\Phi_1$, $\Phi_2$, and $\Phi_3$, of Farey triangles contained in the region $\set{[a,\,b]\in\reals^2:a\ge-1,\,b\ge1}$.
\end{cor}

The cones spanned by $\g$-vectors of maximal sets of pairwise compatible cluster variables in particular constitute a fan.
This fan, as it intersects the plane ${x+y+z=1}$, is illustrated in Figure~\ref{markov g fig}.
The $\g$-vectors are shown as blue dots.
(Cf. \cite[Figure~1]{FG2011} and \cite[Figure~2]{Najera}.)
\begin{figure}[ht]
\scalebox{1}{\includegraphics{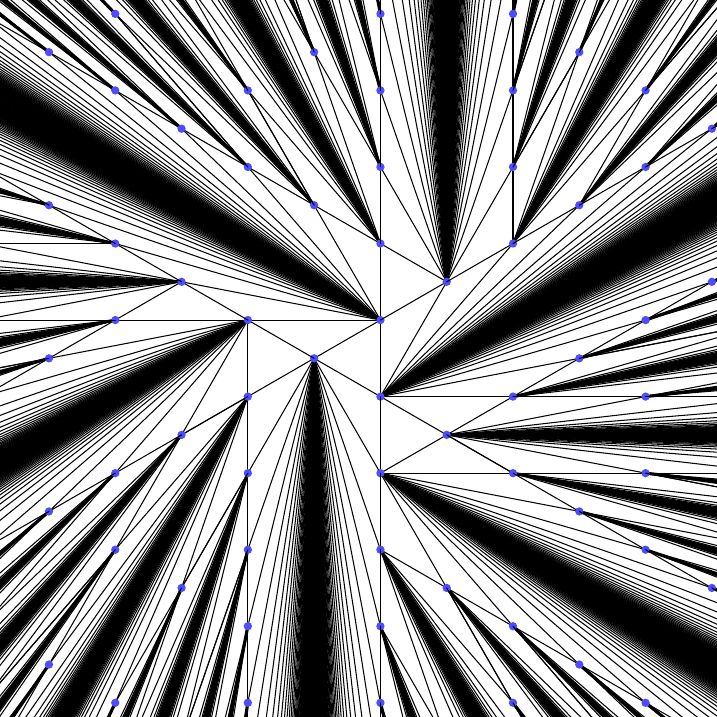}}
\caption{The $\g$-vector fan for $B^T$ as in Corollary~\ref{markov g}}
\label{markov g fig}
\end{figure}

\begin{remark}\label{rescale remark}
An exchange matrix $B'$ is a \newword{rescaling} of an exchange matrix $B$ if $B'=\Sigma^{-1} B\Sigma$ for some diagonal matrix $\Sigma$ with positive entries.
Up to symmetry, the rescalings of $B=\begin{bsmallmatrix*}[r]0&2&-2\\-2&0&2\\2&-2&0\\\end{bsmallmatrix*}$ are 
$B'=\begin{bsmallmatrix*}[r]0&4&-2\\-1&0&1\\2&-4&0\\\end{bsmallmatrix*}$ and $B'=\begin{bsmallmatrix*}[r]0&4&-4\\-1&0&2\\1&-2&0\\\end{bsmallmatrix*}$.
In \cite[Section~7]{universal}, it is explained how to obtain the mutation fan $\F_{B'}$, $R$-bases for $B'$, $\g$-vectors for $(B')^T$, and so forth from the corresponding constructions for $B$, when $B'$ is a rescaling of $B$.
For example, Theorem~\ref{torus FB} combines with \cite[Proposition~7.8(3)]{universal} to yield the following statement.
\begin{cor}\label{torus FB rescale 1}
Let $B'=\begin{bsmallmatrix*}[r]0&4&-2\\-1&0&1\\2&-4&0\\\end{bsmallmatrix*}$.
Then $\F_{B'}$ consists of the following cones and their faces:
\begin{enumerate}
\item \label{torus FB pmO rescale 1}
The nonnegative and nonpositive orthants $\pm O$.
\item \label{torus FB triangle rescale 1}
For each $i=1,2,3,4,5,6$, the cone obtained by applying the map $[x,\,y,\,z]\mapsto[x,\,2y,\,z]$ to the image under $\Phi_i$ of all Farey triangles contained in the region $\set{[a,\,b]\in\reals^2:a\ge-1,\,b\ge1}$.
\item \label{torus FB ray rescale 1}
For each $i=1,2,3,4,5,6$, the cone obtained by applying the map $[x,\,y,\,z]\mapsto[x,\,2y,\,z]$ to the image under $\Phi_i$ of all Farey rays whose vertex is $[a,\,b]$ with $a\ge0$ and $b\ge 1$.
\item
All rays contained in the plane $2x+y+2z=0$.
\end{enumerate}
All of the cones described above are maximal except for the rational rays in the plane $2x+y+2z=0$, each of which is a proper face of the image of two Farey rays.
\end{cor}
\end{remark}

\section{A positive $\reals$-basis}\label{real sec}
In this section, we complete the proof of Theorem~\ref{markov main} by proving the assertion about universal geometric coefficients over $\reals$.
Theorem~\ref{torus FB} implies that the vectors listed in Theorem~\ref{markov main} (1), (2) and (3$'$) consist of one nonzero vector $\v_\rho$ in each ray $\rho$ of $\F_B$.
By Proposition~\ref{pos reals FB converse}, to show that these vectors are a $\reals$-basis for $B$, we need only verify that they are a $\reals$-independent set for $B$. 
The assertion in Theorem~\ref{markov main} then follows by Theorem~\ref{basis univ}.

We begin by realizing the irrational rays of $\F_B$ within the framework of curves in $(\S,p)$.
For each real or infinite slope $\sigma$, choose a line in $\reals^2$ with slope $\sigma$, not containing any integer points.
(Fixing $\sigma$, the line is determined by its $x$-intercept.
Thus, to avoid integer points, there are only countably many $x$-intercepts that must be avoided.)
Let $\lambda(\sigma)$ be the projection of this line to the once-punctured torus.
If $\sigma$ is a rational number or $\infty$ with standard form $b/a$, then $\lambda(\sigma)=\cl(a,b)$.
Otherwise $\lambda(\sigma)$ is dense in the torus.

The \newword{normalized shear coordinates} $\overline{\b}(T_0,\lambda(\sigma))=(\overline{b}_\gamma(T_0,\lambda(\sigma)):\gamma\in T_0)$ of $\lambda(\sigma)$ with respect to $T_0$ are defined as follows:
Start at any point on $\lambda(\sigma)$, fix a direction on the curve and mark off an arc length $d$ in that direction.
(The arc length $d$ is calculated in the usual infinitesimal metric on the unpunctured torus, inherited from $\reals^2$.)
The marked off segment of the curve has shear coordinates given by the rule shown in Figure~\ref{shear fig}.
For large enough $d$, these shear coordinates are nonzero and so can be normalized to unit length in the usual norm on $\reals^3$.
The normalized shear coordinates of $\lambda(\sigma)$ are given by the limit of these unit vectors, as $d$ goes to infinity.
The discussion below in particular shows that this limit exists.

Suppose $\sigma$ is positive or infinite.
Consider a segment of the line with slope $\sigma$ from $y=0$ to $y=k$, and associate a word in $\r$ and $\t$ to this segment as in the proof of Proposition~\ref{explicit coords}.
This word starts and ends with $\t$ and has $k+1$ instances of $\t$ and some number $\ell$ of instances of $\r$.
If $\sigma\ge1$, then this word has no two consecutive letters $\r$.
Thus the word has $\r\t$ $\ell$ times, has $\t\r$ $\ell$ times, and has $\t\t$ $k-\ell$ times.
The shear coordinates of the segment are 
\[\ell[-1,\,0,\,0]+\ell[0,\,1,\,0]+(k-\ell)[-1,\,0,\,1]=[-k,\,\ell,\,k-\ell].\]
If $\sigma\le1$, then the word has no two consecutive letters $\t$.
Thus it has $\r\t$ $k$ times, $\t\r$ $k$ times, and $\r\r$ $\ell-k$ times so the shear coordinates of the segment are 
\[k[-1,\,0,\,0]+k[0,\,1,\,0]+(\ell-k)[0,\,1,\,-1]=[-k,\,\ell,\,k-\ell].\]
As $k$ goes to infinity, the ratio $k/\ell$ goes to $\sigma$, so the limit of the unit vectors in the direction $[-k,\,\ell,\,k-\ell]$ is the unit vector in the direction $[-\sigma,\,1,\,\sigma-1]$.

Nonpositive slopes are dealt with by symmetry as in the proof of Proposition~\ref{explicit coords}.
Specifically, if $-1<\sigma\le0$ then the map $[a,\,b]\mapsto [-b,\,a+b]$ takes $[1,\,\sigma]$ to $[-\sigma,\,\sigma+1]$, which corresponds to a slope $\frac{-\sigma-1}{\sigma}$.
If $\sigma\le-1$ then the map $[a,\,b]\mapsto[-a-b,\,a]$ takes $[1,\sigma]$ to $[-1-\sigma,1]$, corresponding to a slope $\frac{-1}{1+\sigma}$.
We see that the normalized shear coordinates of curves $\lambda(\sigma)$ are the cyclic permutations of the unit vectors in the direction of $[-\sigma,\,1,\,\sigma-1]$ for positive real numbers $\sigma$ or $\sigma=\infty$.
(In the latter case, the unit vector is $\frac1{\sqrt{2}}[-1,\,0,\,1]$.)
This calculation shows in particular that the normalized shear coordinates of $\lambda(\sigma)$ depend only on $\sigma$ and not on the particular choice of a line of slope sigma.
When $\sigma$ is rational or infinite with standard form $b/a$, the normalized shear coordinates of $\lambda(\sigma)$ are obtained from the shear coordinates of $\cl(a,b)$ by normalizing to unit length.

As part of the proof of Proposition~\ref{explicit coords}, we showed that all nonzero integer vectors $[x,\,y,\,z]$ with $x+y+z=0$ such that $x$, $y$, and $z$ have no common factors appear as cyclic permutations of vectors $[-b,\,a,\,b-a]$ such that $b/a$ is the standard form of a rational slope or $\infty$.
The same proof shows that the map $\sigma\mapsto\b(T,\lambda(\sigma))$ is a bijection from real or infinite slopes to unit vectors in the plane $x+y+z=0$.
Thus by Theorem~\ref{torus FB}, the rays of the mutation fan $\F_B$ are spanned by the shear coordinates of curves $\cw(a,b)$ and $\ccw(a,b)$ and by the normalized shear coordinates of curves $\lambda(\sigma)$.

The discussion above establishes that normalized shear coordinates of projected lines exist with respect to the fixed triangulation $T$, but the analogous procedure defines normalized shear coordinates of projected lines exist with respect to any triangulation.

These considerations also suggest the definition of a \newword{real tangle} in the once-punctured torus.
This is a finite collection of distinct curves of the form $\ccw(a,b)$ or $\cw(a,b)$ for standard Farey points $[a,\,b]$ or $\lambda(\sigma)$ for real or infinite slopes $\sigma$.
Each curve is given a real weight.
The shear coordinates of a real tangle are the weighted sum of (normalized) shear coordinates of the curves in the tangle.
A real tangle is null if its shear coordinates are zero with respect to every triangulation.

Proposition~\ref{null tangle is B coher rel} extends to real null tangles, as we now explain.
To begin with, \cite[Theorem~3.9]{unisurface} (which is part of \cite[Theorem~13.5]{cats2}) states that when a triangulation $T$ is altered by a flip, the shear coordinate vector of an allowable curve $\lambda$ is altered by the corresponding mutation map.
Since the normalized shear coordinates of curves $\lambda(\sigma)$ are limits of normalizations of shear coordinates of allowable curves, and since mutation maps are continuous, \cite[Theorem~3.9]{unisurface} extends to curves $\lambda(\sigma)$.
The extension of Proposition~\ref{null tangle is B coher rel} parallels the proof given in \cite[Proposition~7.9]{unisurface} for (integer) null tangles:
The extension of \cite[Theorem~3.9]{unisurface} implies that, for any triangulation $T$, a real null tangle corresponds to a linear relation among normalized shear coordinates that is preserved under all mutation maps.
Then \cite[Proposition~2.3]{unisurface} says that a real null tangle corresponds to a $B(T)$-coherent linear relation among normalized shear coordinates.

Proposition~\ref{one positive or one negative} also extends to real tangles.
The proof is a straightforward concatenation of the extension of Proposition~\ref{null tangle is B coher rel} with \cite[Proposition~4.12]{universal}.
The latter is a version of Proposition~\ref{one positive or one negative} which holds for all $B$-coherent linear relations.

The extension of  Proposition~\ref{null tangle is B coher rel} to real tangles means that we can complete the proof of Theorem~\ref{markov main} by establishing the following ``Real Null Tangle Property'' of the once-punctured torus.
(We emphasize, however, that we have no definition, for general surfaces, of a real tangle.)

\begin{theorem}\label{realnulltangle}
Every real null tangle in $(\S,p)$ is trivial.   
\end{theorem}

To prove the theorem, we first prove several lemmas.

\begin{lemma}\label{slope ineq lemma}
Suppose $q_1<q_2<q_3$ are slopes with standard forms $b/a=q_1$, $d/c=q_2$, and $f/e=q_3$.
Suppose also that the corresponding arcs $\alpha_1$, $\alpha_2$, and $\alpha_3$ constitute a triangulation~$T$ of $(\S,p)$.
Then a curve $\lambda(\sigma)$ has $\overline{b}_{\alpha_3}(T,\lambda(\sigma))>0$ if and only if $\frac{b+d}{a+c}<\sigma<\frac{f}{e}$.
\end{lemma}
\begin{proof}
The situation is illustrated in Figure~\ref{slopeineqfig}.
Two preimages of $\alpha_1$ are drawn in red, two preimages of $\alpha_2$ are drawn in blue, and one preimage of $\alpha_3$ is drawn in purple.
\begin{figure}[ht]
\scalebox{.9}{\includegraphics{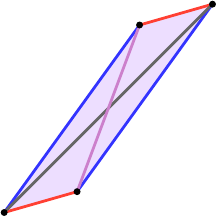}}
\caption{Illustration for the proof of Lemma~\ref{slope ineq lemma}}
\label{slopeineqfig}
\end{figure}
The arc associated to $\frac{b+d}{a+c}$ is drawn gray.
(For readers viewing the article in black and white:  The purple arc is the shorter of the two diagonals of the parallelogram.)
It is immediate that $\frac{b+d}{a+c}<\sigma<\frac{f}{e}$ is a necessary condition for $\overline{b}_{\alpha_3}(T,\lambda(\sigma))$ to be positive.
On the other hand, if $\frac{b+d}{a+c}<\sigma<\frac{f}{e}$, then there exists a line with slope $\sigma$ that intersects the parallelogram in such a way as to have a positive normalized shear coordinate.
The normalized shear coordinates depend only on $\sigma$, so $\overline{b}_{\alpha_3}(T,\lambda(\sigma))>0$.
\end{proof}

\begin{lemma}\label{diophantine lemma}
Given a rational number $q$, there are at most finitely many pairs of Farey points $[a,\,b]$ and $[c,\,d]$ that are Farey neighbors and that have $b/a<q<d/c$.
\end{lemma}
\begin{proof}
Suppose $[a,\,b]$ and $[c,\,d]$ are Farey neighbors with $b/a<q<d/c$.
In particular, $a$ and $c$ weakly agree in sign, so it is enough to show that there are finitely many possibilities with $a\ge0$ and $c\ge0$.
Proposition~\ref{torus arcs compatible} says that $ad-bc=1$.

If $a=0$, then since $b/a<q$, we have $[a,\,b]=[0,\,-1]$.
Thus $ad-bc=c=1$ and $d\le 0$.
There are at most finitely many integers $d\le 0$ such that $d>q$.
Similarly, if $c=0$ then $[c,\,d]=[0,\,1]$, so that $a=1$ and $b\ge 0$.
There are finitely many integers $b\ge0$ such that $b<q$.

Now suppose $a>0$ and $c>0$.
Rewrite the identity $ad-bc=1$ as $b/a=d/c-1/(ac)$, so that $d/c-1/(ac)<q<d/c$ and multiply through by $c$ to obtain $d-1/a<cq<d$.
For $a$ large enough, there are no multiples of $q$ strictly between an integer $d$ and the quantity $d-1/a$, because $q$ is rational.
Thus there are only finitely many possible values of $a$.

Now use the identity $b/a=d/c-1/(ac)$ to rewrite the inequality $b/a<q<d/c$ as $b/a<q<b/a+1/(ac)$ and multiply through by $a$ to obtain $b<aq<b+1/c$.
We have $1/c\le1$, and so for each fixed $a$ there is at most one $b$ such that $b<aq<b+1/c$.
We have shown that there are only finitely many possible points $[a,\,b]$.

The symmetric argument shows that there are at most finitely many points $[c,\,d]$.
Specifically, for large enough $c$, the inequality $b<aq<b+1/c$ cannot be satisfied.
Also, for fixed $c$, there is at most one $d$ satisfying $d-1/a<cq<d$.
\end{proof}

\begin{lemma}\label{tech lemma}
Given real numbers $x<y$ with $y$ irrational, there exists a Farey triangle with vertices $[a,\,b]$, $[c,\,d]$, and $[e,\,f]$ having $b/a<d/c<f/e$ such that $x<\frac{b+d}{a+c}<y<\frac{f}{e}$.
\end{lemma}
\begin{proof}
We will find $s>0$ such that the vector $s[1,\,y]$ is in such a Farey triangle.
Since $y$ is irrational, no vector $s[1,\,y]$ is in a Farey ray.
The symmetric statement to Lemma~\ref{Farey decomp} (switching $a$ and $b$) implies that for $s>1$, the vector $s[1,\,y]$ is in some Farey triangle.
Choose a rational number $q$ strictly between $x$ and $y$.
Lemma~\ref{diophantine lemma} implies that there are only finitely many Farey triangles having a vertex whose corresponding slope is less than $q$ \emph{and} having a vertex whose corresponding slope is greater than $q$.
Thus there is some number $s_{\lim}>1$ such that for $s>s_{\lim}$, the vector $s[1,\,y]$ is in some Farey triangle all of whose vertices correspond to slopes greater than or equal to $q$.

Let $\Delta_0$ be some Farey triangle containing $s_0[1,\,y]$ for some $s_0>s_{\lim}$.
As $s$ increases from $s_0$, the point $s[1,\,y]$ passes through a sequence $\Delta_0,\Delta_1,\Delta_2,\ldots$ of Farey triangles.
We claim that there exists $k>0$ such that, writing $[a,\,b]$, $[c,\,d]$, and $[e,\,f]$ for the vertices of  $\Delta_k$ with $b/a<d/c<f/e$, the triangle $\Delta_{k+1}$ shares the vertices $[c,\,d]$ and $[e,\,f]$ with $\Delta_k$.
If not, then for any $k\ge0$, the vector $s[1,\,y]$ passes through infinitely many edges defined by Farey neighbors one of which is $[a,\,b]$ and the other of which corresponds to a slope greater than $y$.
This gives a contradiction to Lemma~\ref{diophantine lemma} for any rational number $q'$ with $b/a<q'<y$.
This contradiction proves the claim.

For $\Delta_k$ as in the claim, we have $b/a<d/c<y<f/e$.
Also, by construction $x<q\le b/a$.
We can complete the proof by showing that $b/a<(b+d)/(a+c)<d/c$.
Since $x<b/a$, we have $[a,\,b]\neq[0,\,-1]$, and since $d/c>f/e$, we also have $[c,\,d]\neq[0,\,1]$.
Thus both $a$ and $c$ are positive, so the inequality $b/a<(b+d)/(a+c)$ follows easily from $b/a<d/c$.
Similarly, $(b+d)/(a+c)<d/c$.
\end{proof}

\begin{proof}[Proof of Theorem~\ref{realnulltangle}]
Let $\Xi$ be a real null tangle that is not trivial.
The proof of Proposition~\ref{null tangle lines only} shows that $\Xi$ is supported on curves of the form $\lambda(\sigma)$ (using the extension of Proposition~\ref{one positive or one negative} to real tangles).
The remainder of the proof of Theorem~\ref{torus null tangle} shows that $\Xi$ is supported on curves $\lambda(\sigma)$ for $\sigma$ irrational (using the extension of Proposition~\ref{null tangle is B coher rel} to real tangles).
The support of $\Xi$ has at least two curves.
Let $\sigma_1$ and $\sigma_2$ be the largest two numbers indexing curves in the support of $\Xi$, with $\sigma_1<\sigma_2$.
Lemma~\ref{tech lemma} says that there exists a Farey triangle with vertices $[a,\,b]$, $[c,\,d]$, and $[e,\,f]$ with $b/a<d/c<f/e$ such that $\sigma_1<\frac{b+d}{a+c}<\sigma_2<\frac{f}{e}$.
By Propositions~\ref{torus arcs} and~\ref{torus arcs compatible}, the arcs given by slopes $b/a$, $d/c$, and $f/e$ form a triangulation $T$ of $(\S,p)$.
Let $\gamma$ be the arc associated to $f/e$.
Lemma~\ref{slope ineq lemma} and our choice of $\sigma_1$ and $\sigma_2$ imply that $\lambda(\sigma_1)$ is the unique curve in the support of $\Xi$ whose normalized shear coordinates with respect to $T$ have a positive entry in the position indexed by $\gamma$.
By the extension of Proposition~\ref{one positive or one negative}, we conclude that $\lambda(\sigma_1)$ appears with coefficient $0$ in $\Xi$, and this contradiction proves the theorem.
\end{proof}

%SinceSLC:  Anything I add about denominator vectors comes after SLC submission!

\section{Denominator vectors}\label{denom sec}
We close with an aside about denominator vectors ($\d$-vectors).
Denominator vectors for marked surfaces are given by intersection numbers, as described in \cite[Theorem~8.6]{cats1}.
These intersection numbers are easy to compute in the case of the once-punctured torus. 
Specifically, Proposition~\ref{torus arcs} combines with \cite[Theorem~8.6]{cats1} to prove the following proposition, which was already pointed out in \cite[Example~2.18]{NS} with the same proof.

\begin{prop}\label{markov d}
Let $B=\begin{bsmallmatrix*}[r]0&2&-2\\-2&0&2\\2&-2&0\\\end{bsmallmatrix*}$.
Then the $\d$-vectors of cluster variables associated to $B$ are the cyclic permutations of vectors $[a-1,b-1,a+b-1]$, for all (possibly infinite) positive rational slopes with standard form $b/a$.
\end{prop}
%In particular, all of the $\d$-vectors are contained in the hypersurface defined by the symmetric polynomial
%\begin{multline*}
%(x^3+y^3+z^3)-(x^2y+x^2z+xy^2+xz^2+y^2z+yz^2)+2xyz\\+(x^2+y^2+z^2)-2(xy+xz+yz)-(x+y+z)-1.
%\end{multline*}
% BUT THAT'S OBVIOUS because this polynomial is the product of the linear forms defining the three obvious planes!
Proposition~\ref{torus arcs compatible} allows us to picture the $\d$-vector fan as well.
Two stereographic views are shown in Figure~\ref{denoms}.
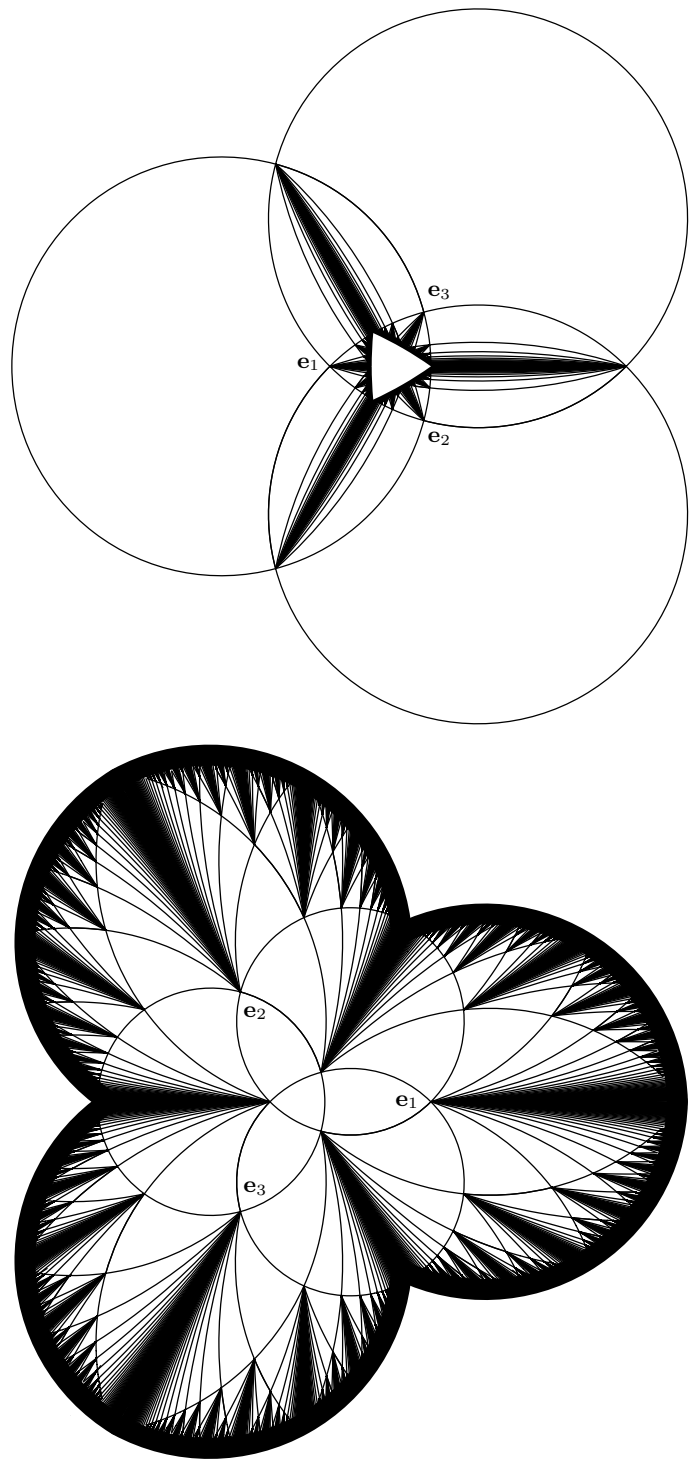
\begin{figure}[p]
\scalebox{0.97}{\includegraphics{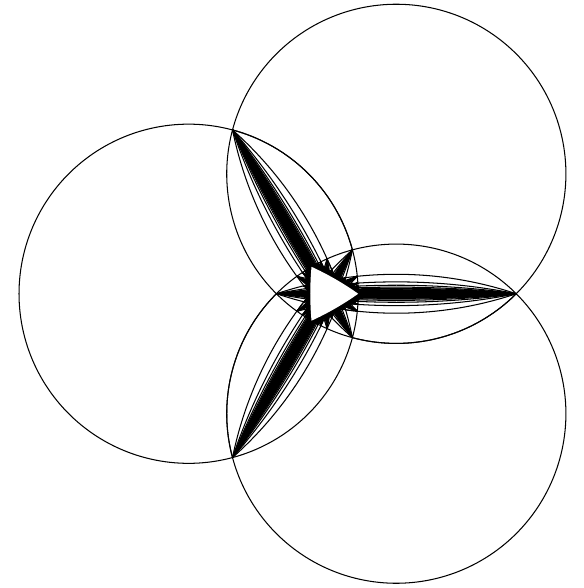}
\begin{picture}(0,0)(143,-141)
\put(30,-29){\small$\e_2$}
\put(30,28){\small$\e_3$}
\put(-21,0){\small$\e_1$}
\end{picture}}

\scalebox{0.99}{\includegraphics{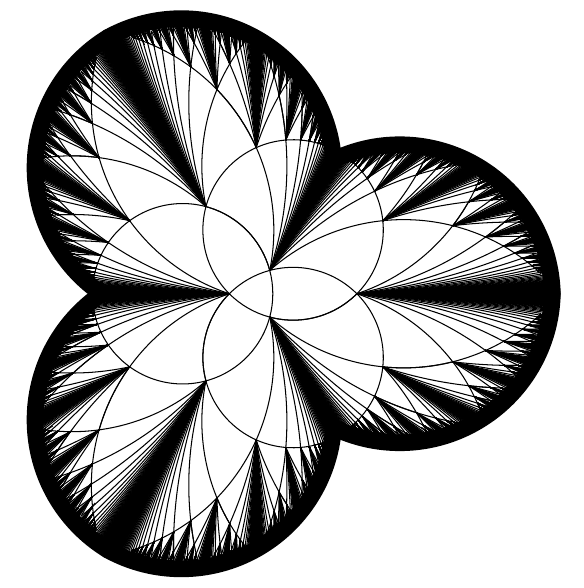}
\begin{picture}(0,0)(138,-141)
\put(-46,33){\small$\e_2$}
\put(-46,-34){\small$\e_3$}
\put(12,-1){\small$\e_1$}
\end{picture}}
\caption{Two views of the $\d$-vector fan for the once-puntured torus}
\label{denoms}
\end{figure}

\section{Extensions}\label{ext sec}
A natural question is to what extent the results of this paper can be extended to other surfaces or to other cluster algebras.
One would expect that the natural extension is to once-punctured surfaces of higher genus.
But it is difficult to see how to proceed in higher genus.
The key to the results for the torus is the fact that curves are indexed by rational slopes and that compatibility of curves is described by the Farey condition.
In higher genus, curves and compatibility appear to be much more complicated.

However, there is another surface where rational slopes and Farey-type conditions describe curves and compatibility:  the four-punctured sphere.
In~\cite{unisphere}, Barnard, Meehan, Polster, and the author extend the results of this paper to that surface.
Some of the preliminary results from \cite{unisphere} are already present in \cite[Section~5]{tubular1}.
Indeed, the four-punctured sphere is the simplest of the 4 \newword{tubular cluster algebras} discussed in \cite{tubular1,tubular2} and is the only one of the 4 associated to a surface.
Given the prominence of rational slopes and the Farey condition in \cite{tubular1,tubular2}, it seems plausible that the results of this paper (except as they pertain to marked surfaces) extend in some form to the other 3 tubular cluster algebras as well.
It also seems reasonable to wonder to what extent the once-punctured torus itself can be treated as a tubular cluster algebra.

\subsection*{Acknowledgments}
Thanks to Emily Barnard, Emily Meehan, Shira Polster, Salvatore Stella, and two anonymous referees for making valuable corrections to an earlier version.

%\addtocontents{toc}{\mbox{ }}

\end{document}